\documentclass[11pt]{amsart}

\usepackage{verbatim, amssymb,hyperref}
\usepackage{amsmath}
\usepackage{color}
\usepackage{bbm}
\usepackage{graphicx}
\usepackage{subfig}
\usepackage[labelformat=empty, font=scriptsize, labelfont=bf, width = .25\linewidth]{caption}

\newcommand{\black}{\color[rgb]{0,0,0}}
\newcommand{\red}{\color[rgb]{1,0,0}}

\newcommand{\eps}{\varepsilon}

\newcommand{\cF}{\mathcal F}
\newcommand{\cK}{\mathcal K}
\newcommand{\cC}{\mathcal C}
\newcommand{\cL}{\mathcal L}

\newcommand{\cO}{\mathcal O}

\newcommand{\var}{\mathrm{Var}}

\newcommand{\grad}{\nabla}

\newcommand{\sfrac}[2]{\mbox{$\frac{#1}{#2}$}}

\newcommand{\rd}{d}

\newcommand{\loja}{\L ojasiewicz}
\newcommand{\CL}{\L}

\newcommand{\IB}{\mathbb B}

\newcommand{\II}{\mathbb I}

\newcommand{\cov}{\mathrm{cov}}

\newcommand{\1}{1\hspace{-0.098cm}\mathrm{l}}

\renewcommand{\P}{{\mathbb P}}

\newcommand{\N}{{\mathbb N}}

\newcommand{\IL}{{\mathbb L}}
\newcommand{\IC}{{\mathbb C}}
\newcommand{\E}{{\mathbb E}}

\newcommand{\R}{{\mathbb R}}

\newcommand{\Hess}{\text{Hess}\,}

\newcommand{\tr}{\mathrm{tr}}

\theoremstyle{plain}
\newtheorem{theorem}{Theorem}[section]
\newtheorem{prop}[theorem]{Proposition}
\newtheorem{lemma}[theorem]{Lemma}
\newtheorem{cor}[theorem]{Corollary}
\newtheorem{defi}[theorem]{Definition}

\newtheorem*{nota}{Notation}

\theoremstyle{definition}
\newtheorem{rem}[theorem]{Remark}

\makeatletter
\@namedef{subjclassname@2020}{%
	\textup{2020} Mathematics Subject Classification}
\makeatother

\begin{document}
	
	\title[Convergence of SDE]%
	{Cooling down stochastic differential equations: \\ almost sure convergence}
	
	\author[]
	{Steffen Dereich}
	\address{Steffen Dereich\\
		Institut f\"ur Mathematische Stochastik\\
		Fachbereich 10: Mathematik und Informatik\\
		Westf\"alische Wilhelms-Universit\"at M\"unster\\
		Orl\'eans-Ring 10\\
		48149 M\"unster\\
		Germany}
	\email{steffen.dereich@wwu.de}
	
	\author[]
	{Sebastian Kassing}
	\address{Sebastian Kassing\\
		Institut f\"ur Mathematische Stochastik\\
		Fachbereich 10: Mathematik und Informatik\\
		Westf\"alische Wilhelms-Universit\"at M\"unster\\
		Orl\'eans-Ring 10\\
		48149 M\"unster\\
		Germany}
	\email{sebastian.kassing@wwu.de}
	
	
	\begin{abstract}
		We consider almost sure convergence of the SDE 
		$
		\rd X_t=\alpha_t \, \rd t + \beta_t \, \rd W_t
		$
		under the existence of a $C^2$-Lyapunov function $F:\R^d \to \R$. 
		More explicitly, we show that on the event that the process stays local we have almost sure convergence in the Lyapunov function $(F(X_t))_{ t \ge 0}$ as well as $\nabla F(X_t)\to 0$, if $|\beta_t|=\cO( t^{-\beta})$ for  a $\beta>1/2$. If, additionally, one assumes that $F$ is a \CL ojasiewicz-function, we get almost sure convergence of the process itself, given that $|\beta_t|=\cO(t^{-\beta})$ for a  $\beta>1$. The assumptions are shown to be optimal in the sense that there is a divergent counterexample where $|\beta_t|$ is of order $t^{-1}$.
		\end{abstract}

	\maketitle
	
	\section{Introduction}
	Let $(\Omega, (\cF_t)_{t \ge 0}, \cF,\P)$ be a filtered probability space satisfying the usual conditions and suppose that $(X_t)_{t\ge0}$ is a continuous semimartingale  satisfying  the equation
	\begin{align}\label{SDE}
	X_t-X_0 = \int_0 ^t  \alpha_s \, ds +\int_0^t  \beta_s \, dW_s,
	\end{align}
	for all $t\ge0$, 
	where
	\begin{itemize}
		\item $(\alpha_t)_{t \ge 0}$ is a progressive,  $\R^d$-valued process,
		\item $(\beta_t)_{t \ge 0}$ is a progressive $\R^{d \times d'}$-valued process and
		\item $(W_t)_{t \ge 0}$ is an $\R^{d'}$-valued $(\cF_t)_{t \ge 0}$-Brownian motion with  initial value $W_0=0$.
	\end{itemize}
	
	In this article, we analyse convergence of  $(X_t)_{t \ge 0}$ under the assumption that the drift term satisfies a Lyapunov kind of condition and the diffusivity converges to zero sufficiently fast.
	More explicitly, we denote by $F: \R^d \to \R$  a $C^2$-function, by $f :=\grad F$ its gradient and by  $H:=\Hess F$ the Hessian of $F$ and consider the following event
	\begin{align*} 
	\IC := \Bigl\{\liminf_{t\to\infty} &\frac{\langle f(X_t),-\alpha_t\rangle}{|f(X_t)|^2} >0 ,\, \liminf_{t \to \infty}\frac {|f(X_t)|}{|\alpha_t|}>0, \\
	& \quad \int_0^\infty |\beta_s|_F^2 \, \rd s<\infty \, \text{ and}\	\limsup\limits_{t \to \infty } |\beta_t| < \infty \Bigr\},
	\end{align*}
	where we interpret $\frac 00$ as  $\infty$  and, for a matrix $A \in \R^{d \times d'}$, $|A|_F$ (resp. $|A|$) denotes the Frobenius norm (resp. spectral norm).
	Intuitively, on $\IC$ the drift term $(\alpha_t)_{t \ge 0}$ is comparable in size and direction to the negative gradient of the Lyapunov function, i.e.,  $(-\nabla F(X_t))_{t \ge 0}$, and the size of the diffusivity $(\beta_t)_{t \ge 0}$ can be controlled at late times. The first and second condition in the definition of the event $\IC$ are obviously satisfied if we consider the solution to the SDE
		\begin{align} \label{eq:SGFlow}
		\rd X_t = -\nabla F(X_t) \, \rd t + \beta_t \, \rd W_t,
		\end{align}
		which represents the canonical stochastic version of the gradient flow ODE
		$$
			\dot x_t =-\grad F(x_t).
		$$
		However, the current framework allows more flexibility in the drift term. 
	We will restrict attention to certain events: set
	$$ 
	\IB:=\Biggl\{ \liminf_{t\to\infty} F(X_t)>-\infty ,\,  \limsup_{t \to \infty} |\alpha_t|<\infty \text{ and }  \limsup_{\eps\to0} \limsup_{\substack{t\to\infty \\ y\in B(X_t,\eps)}}  |H(y)| <\infty \Biggr\}
	$$
	and
	$$
	\IL:=\Bigl\{\limsup_{t\to\infty} |X_t|<\infty \Bigr\}.
	$$
	We separately show  almost sure convergence of $(F(X_t))_{t \ge 0}, (f(X_t))_{t \ge 0}$ and $(X_t)_{t \ge 0}$ 
	under mild assumptions on the diffusivity and the underlying Lyapunov function which are shown to be optimal (c.f.\ Chapter~\ref{sec:optimal}).   The SDE (\ref{eq:SGFlow}) is  the continuous-time counterpart for the stochastic gradient descent scheme (when $(\beta_t)_{t \ge 0}$ is chosen properly, see \cite{WeinanE2017}) and, thus, a helpful tool for gaining intuition and insight on the behaviour of SGD. It also appears naturally in applications from Physics and Chemistry. 
		For example, (\ref{eq:SGFlow}) can be viewed as the pathwise solution to the Fokker--Planck equation (see \cite{gardiner1985handbook}) in a physical system, where  $F$ models the energy and $(\beta_t)_{t \ge 0}$ an external source of perturbation\footnote{$(2/\beta_t^2)_{t \ge 0}$ is sometimes referred to as the \emph{inverse temperature}, e.g.~\cite{chiang1987diffusion, jordan1998variational}}, or as the chemical Langevin equation describing a chemical reaction evolving in time (\cite{van1992stochastic}, \cite{gillespie2000chemical}, \cite{schnoerr2014complex}). In our analysis, we allow general non-convex energy landscapes and focus on the case where the perturbation fades within time. 

There are several asymptotic results for the solution $(X_t)_{t \ge 0}$ of (\ref{eq:SGFlow}) when choosing a  deterministic diffusivity $(\beta_t)_{t \ge 0}$.  
Inspired by simulated annealing, \cite{geman1986diffusions} and \cite{chiang1987diffusion} proved that for $(\beta_t)_{t \ge 0}=(\sqrt{2c/ \log (t+1)})_{t \ge 0}$ with sufficiently large $c>0$ and under additional assumptions on the target function $F$, the solution of (\ref{eq:SGFlow}) converges weakly to a distribution $\pi$ concentrated on the global minima of $F$ independent of the initial value $X_0=x \in \R^d$, even in the case of non-convex target functions $F$. More precisely, $\pi$ is the weak limit of the Gibbs densities $\pi_t(x)\propto \exp(-2F(x)/\beta^2_t)$. In this regime, \cite{holley1989asymptotics}  showed that if the diffusion process is defined on a compact manifold, then $(F(X_t))_{t \ge 0}$ converges to the optimal value in probability if and only if $c>c^*$, where $c^*$ denotes the minimal potential energy necessary to connect every point with a global minimum using a continuous path. Recently, this result has been generalised by \cite{fournier2021simulated} to diffusions 
on $\R^d$ in the case where $F \in C^\infty$ is assumed to satisfy $\lim_{|x| \to \infty} F(x)=\infty$ and $\int_{\R^d} e^{-\alpha_0 F(x)} \, \rd x<\infty$ for some $\alpha_0>0$. In that case one has $\liminf_{t \to \infty} |X_t|<\infty$, almost surely.

However, in this particular scenario we cannot hope for the individual paths of the solution $(X_t)_{t \ge 0}$ to converge. Conversely, the diffusivity $(\beta_t)_{t \ge 0}$ considered here converges significantly faster to zero. Therefore, we cannot guarantee that the process $(X_t)_{t \ge 0}$ started in the basin of attraction of a bad local minimum is able to overcome a barrier in the landscape in order to converge to a global minimum (see Section~\ref{sec33}) and the distribution of the limit points of $(X_t)_{t \ge 0}$ depends heavily on the distribution of the initial value $X_0$. 
As we will indicate below, the gradient flow converges if $F$ is a  \loja-landscape, 
in the sense of Definition~\ref{def:loja}, and the gradient flow does not escape to infinity. The main contribution of this article is the analysis in how far this is still the case when additionally incorporating a stochastic perturbation.
We stress that we analyse the asymptotic behaviour of the system as time tends to infinite which is significantly different from the analysis of families of SDEs where the diffusivity decays with a family index $\eps>0$, see e.g.\ \cite{freidlin2012random}.
%
	
	First, we will provide a short proof for almost sure convergence of $(F(X_t))_{t \ge 0}$ and $(f(X_t))_{t \ge 0}$ with  $\lim_{t\to\infty}f(X_t)=0$ 
	 on the event $\IC\cap \IB$. We conclude that we have a.s.\ convergence of $(X_t)_{t \ge 0}$ on the event $\IC \cap \IL$ in the case where where the set of critical points $\cC:=\{x \in \R^d:f(x)=0\}$ is at most countable.
%
%
	\begin{theorem}\label{theo1}
		On $\IC\cap \IB$, almost surely, $(F(X_t))_{t\ge0}$ converges to a finite value and $\lim_{t\to\infty} f(X_t)=0$ as well as $\limsup_{t_0 \to \infty } \int_{t_0}^\infty \langle f(X_s),-\alpha_s\rangle\,\rd s<\infty$.
		If the set of critical points of $F$ 
		is at most countable, then,  on $\IC \cap \IL$, almost surely, $(X_t)_{t \ge 0}$ converges.
	\end{theorem}

	Theorem~\ref{theo1} is an immediate consequence of the more general Proposition~\ref{prop:conv1} below.
	
	\begin{rem}
		Note that as $F$ is $C^2$ we clearly have $\IC \cap \IL \subset \IC \cap \IB$ so that the first statement of Theorem~\ref{theo1} also holds on $\IC \cap \IL$. Moreover, for a Lyapunov function $F$ with $\liminf_{|x| \to \infty} |f(x)|>0$ the first statement implies that
		$$
		\IC \cap \IB = \IC \cap \IL.
		$$
	\end{rem}

	\begin{rem} 
		A standard approach for proving convergence statements similar to those given in Theorem~\ref{theo1} is the theory of asymptotic pseudotrajectories (see \cite{benaim1999dynamics}, especially Corollary~6.7). While these methods may lead to different and typically weaker assumptions on the perturbation $(\beta_t)_{t \ge 0}$, Theorem~\ref{theo1} will mainly be used to prove Theorem~\ref{theo2} below. In order to be self-contained we stated Theorem~\ref{theo1} in such a way that it can easily be applied in Theorem~\ref{theo2} and has a direct and brief proof.
	\end{rem}
	
	In the case where the set of critical points of $F$ contains a continuum of elements the situation is more subtle. Even in the case without random perturbations convergence of the solution $(x_t)_{t \ge 0}$ to the ODE
	\begin{align} \label{eq:ODE}
	\dot x_t =-\grad F(x_t)
	\end{align}
	 is a non-trivial issue: one can find $C^\infty$-functions $F$ together with solutions $(x_t)_{t\ge0}$  that stay on compact sets but do not converge,  see example 3 on page 14 of \cite{palis2012geometric}. Counterexamples of this structure have been known for a long time (see e.g.~\cite{curry1944method}) and include the famous Mexican hat function \cite{AMA05}. 
	To guarantee convergence (at least in the case where the solution stays on a compact set), one needs to impose additional assumptions. An appropriate assumption is the validity of a   \loja-inequality, see Definition~\ref{def:loja} below. This assumption has the appeal that it is satisfied by analytic functions, see~\cite{lojasiewicz1963propriete,lojasiewicz1965ensembles}. With the help of the inequality (\ref{eq:loja}) \CL ojasiewicz showed in \cite{lojasiewicz1984sur} that any integral curve of the gradient flow equation (\ref{eq:ODE}) for an analytic function $F$ that has an accumulation point has a unique limit. A proof of that assertion can also be found in \cite{haraux2012some} or \cite{colding2014}.
	
	In this article, we device a probabilistic counterpart to the classical analytic approach. We will derive moment estimates of $F(X_t)$ for  \L ojasiewicz-functions $F:\R^d\to \R$  on $\IC \cap \IL$ under the additional assumption that we can asymptotically bound the diffusivity $(\beta_t)_{t \ge 0}$ by a locally-bounded deterministic function $\sigma: [0,\infty) \to [0,\infty)$ to deduce almost sure convergence of $(\int_0^t \alpha_s \, \rd s)_{t \ge 0}$ and thus, almost sure convergence of $(X_t)_{t \ge 0}$.
	%
	%
	

	\begin{defi}\label{def:loja}
		We call the function $F:\R^d\to\R$ \emph{\loja-function}, if for every $x\in\cC:=f^{-1}(\{0\})$, the \loja-inequality is true on a neighbourhood $U_x$ of $x$ with parameters $\CL>0$ and $\theta\in[\frac 12,1)$, i.e., for all $y\in U_x$
		\begin{align} \label{eq:loja}
		|f(y)|\ge \CL\,|F(y)-F(x)|^\theta.
		\end{align}
	\end{defi}
	
	\begin{rem} 
		Note that for all $x \notin \mathcal C$ and $\theta \in [\frac 12, 1)$ there trivially exist a neighbourhood $U_x$ and constant $\CL>0$ such that the \loja-inequality with parameter $\CL$ and $\theta$ hold on $U_x$.
	\end{rem}
	
	The main contribution of this article concerns convergence of the SDE in the case where the Lyapunov function is a \loja-function.

	\begin{theorem} \label{theo2}
		Let $F$ be a \loja-function and $(\sigma_t)_{t\ge 0}=((t+1)^{-\sigma})_{t \ge 0}$ with $\sigma>1$. 
		Then, on $\IC \cap \IL\cap \bigl\{ \limsup_{t \to \infty} \frac{|\beta_t|}{\sigma_t}<\infty \bigr\}$, the process $(X_t)_{t \ge 0}$ satisfying~(\ref{SDE}) converges, almost surely, to a critical point of $F$. 
	\end{theorem}

	The proof of the theorem is accomplished in Section~\ref{sec3}.
	The theorem is optimal in the sense that its conclusion is not true for the choice $(\sigma_t)_{t \ge 0}=((t+1)^{-1})_{t \ge 0}$.
	More explicitly,  we provide an SDE for which $\IC \cap \IL\cap \bigl\{ \limsup_{t \to \infty} (t+1)   |\beta_t|<\infty \bigr\}$ is an almost sure set and for which $(X_t)_{t \ge 0}$ diverges, almost surely,  in Section~\ref{sec:optimal}.
	
	\begin{rem} 
		Choose $(\alpha_t)_{t \ge 0}=(-f(X_t))_{t \ge 0}$ and consider the solution to the SDE (\ref{eq:SGFlow}). In that case we have for $(\sigma_t)_{t \ge 0}$ as in Theorem~\ref{theo2} that
		$
			\bigl\{ \limsup_{t \to \infty} \frac{|\beta_t|}{\sigma_t}<\infty \bigr\} \subset \IC.
		$
		Now, assume that $F \in C^2$ satisfies $\lim_{|x| \to \infty} F(x)=\infty$. In Appendix~\ref{sec:app} we derive sufficient conditions that guarantee locality of the process $(X_t)_{t \ge 0}$. In particular, we show that if the Hessian of $F$ is uniformly bounded (i.e. $\sup_{x \in \R^d}|H(x)|<\infty$) then $\bigl\{\int_0^\infty |\beta_s|_F^2 \, \rd s<\infty \bigr\} \subset \IL$.
		Also, we get for any compact set $K \subset \R^d$
		\begin{align*}
			\Bigl\{ &\1_{\{X_t \notin K\}} (-|f(X_t)|^2+\frac 12 \tr ((\beta_t)^\dagger H(X_t) \beta_t) )\le 0  \text{ for all large } t \\ 
			&\text{ and } \int_0^\infty \1_{\{X_s \in K\}}|\beta_s|_F^2 \, \rd s <\infty \Bigr\}\subset \IL.
		\end{align*}
	\end{rem}

	 The proof of Theorem~\ref{theo2} is based on  moment estimates for  $F(X_t)$ restricted to particular events, see Proposition~\ref{prop:loja1} below. In the particular case where we approach a local minimum these estimates entail the following moment estimate. 

	\begin{theorem} \label{the:rate} 
		Let $K\subset \R^d$ be a compact set, $\CL>0$, $\theta \in (1/2,1)$, $\ell\in\R$, $C_\beta>0$ and $\sigma\in[ 0,\infty)$. Suppose that  for all $y \in K$
		$$
		F(y)\ge \ell \ \text{and } \ |f(y)|\ge \CL (F(y)-\ell)^\theta.
		$$
		Consider  the stopping time 
		$$
		T:= \inf\Bigl\{ t\ge 0 : X_t \notin K, \frac{\langle f(X_t),-\alpha_t\rangle}{|f(X_t)|^2}\le \rho \ , \ |\beta_t|\ge C_\beta (t+1)^{-\sigma} \Bigr\}.
		$$
		Then there exists $C\ge 0$ so that for all $t\ge 0$
		$$
		\E[\1_{\{T > t\}} (F(X_t)-\ell) ]\le C (t+1)^{-(\frac \sigma\theta \wedge \frac{1}{2\theta-1})}.
		$$
	\end{theorem}

	Theorem~\ref{the:rate} will be a consequence of Proposition~\ref{prop:loja1} and Remark~\ref{rem:loja1} below.  Note that if $\sigma\ge\theta/(2\theta-1)$ we essentially get the same order of convergence as in the case without random perturbations. If the drift terms vanishes at a slower rate, i.e. $\sigma<\theta/(2\theta-1)$, this effects the order of decay of the expected function value $(F(X_t)-\ell)_{t \ge 0}$. 

	\begin{nota}
		For a vector $v \in \R^d$ we denote 
		by $|v|$ the Euclidean norm induced by the scalar product, i.e. $|v|^2 = \langle v,v\rangle$. For a matrix $A\in \R^{d\times d'}$ we denote by $|A|$ the spectral norm
		and by $|A|_F$ the Frobenius norm, i.e.
		$$
		|A|:=\max\limits_{|x|\neq 0}\frac{|Ax|}{|x|}
		\quad \text{ and } \quad  |A|_F:=\sqrt{\mathrm{tr}(A^\dagger A)}.
		$$
		For $t_0 \ge 0$ and a real-valued continuous semimartingale $(Y_s)_{s\ge t_0}$ we denote by $\langle Y\rangle=(\langle Y\rangle_s)_{s\ge t_0}$ its quadratic variation process.
	\end{nota}

	\section{Proof of Theorem~\ref{theo1}}
	 In this section we prove a slight generalisation of Theorem~\ref{theo1}. 
		
	\begin{prop} \label{prop:conv1}
		Consider
				\begin{align*}
		\IC_1:=\Bigl\{&\liminf_{t\to\infty} F(X_t)>-\infty,\ \limsup_{t\to\infty} |\beta_t|<\infty, \  \liminf_{t\to\infty} \frac {\langle f(X_t),-\alpha_t\rangle}{|f(X_t)|^2} >0 \, \text{ and }\\
		& \Bigl(\int_0^{t} \tr(\beta_s^\dagger H(X_s) \beta_s)\,\rd s \Bigr)_{t \ge 0} \text{ converges to a finite value} \Bigr\}
		\end{align*}
		$$
		\IC_2:=\IC_1\cap \Bigl\{\limsup_{t\to\infty} |\alpha_t|<\infty,\, \limsup_{\eps\to0} \limsup_{\substack{t\to\infty \\ y\in B(X_t,\eps)}}  |H(y)| <\infty, \int_0^\infty |\beta_s|_F^2<\infty\Bigr\},
		$$
		and
		$$
		\IC_3:=\IC_2\cap\Bigl\{\liminf_{t\to\infty} |X_t|<\infty\Bigr\},
		$$
		where we interpret $\frac 00$ as $\infty$. On $\IC_1$, almost surely, $(F(X_t))_{t \ge 0}$ converges  to a finite value and $\limsup_{t_0 \to \infty } \int_{t_0}^\infty \langle f(X_s),-\alpha_s\rangle\,\rd s<\infty$.
		On $\IC_2$, almost surely, $\lim_{t\to\infty} f(X_t)=0$.  If the set of critical points of $F$ is at most countable,
		then,  on $\IC_3$, almost surely, $(X_t)_{t \ge 0}$ converges.\footnote{Note that $\IC_3$ entails $\IC_1$ and $\IC_2$.}
	\end{prop} 
 
		Note that $\IC \cap \IB \subset \IC_2$ since $\int_0^\infty |\beta_s|_F^2 \, \rd s<\infty$ and $\limsup_{t \to \infty} |H(X_t)|<\infty$ imply convergence of
		$$
		\Bigl(\int_0^{t} \tr(\beta_s^\dagger H(X_s) \beta_s)\,\rd s\Bigr)_{t \ge 0}.
		$$
		Therefore, Theorem~\ref{theo1} directly follows from Proposition~\ref{prop:conv1}.
%
	
	\begin{proof}[Proof of Proposition~\ref{prop:conv1}]
		Let $t_0, C_1, C_2 \ge 0$ and set
		\begin{align*}
		T:=T^{(t_0,C_1,C_2)}:=\inf \Bigl\{ t\ge t_0: \ &  F(X_t)<-C_1  \, \vee \, \int_{t_0}^t \tr(\beta_t^\dagger H(X_t) \beta_t)\,\rd s>C_1 \, \vee \, |\beta_t|>C_1,\\
		&
		\text{ \ or \ } \langle f(X_t),\alpha_t\rangle >-C_2 |f(X_t)|^2\Bigr\}
		\end{align*}
		

		By Itô's formula one has 
		\begin{align}
		\label{eq:1}
		\rd F(X_t)&= \langle f(X_t) ,\alpha_t\rangle \,\rd t+ \langle f(X_t)),\beta\,\rd W_t\rangle +\frac 12 \tr(\beta_t^\dagger H(X_t) \beta_t)\,\rd t.
		\end{align}
		We let
		$$
		(M_t)_{t \ge t_0}:=\Bigl( \int_{t_0}^{t\wedge T}  \langle f(X_t),\beta\,\rd W_t\rangle \Bigr)_{t\ge t_0}
		$$
		and note that its variation process satisfies
		$$\langle M\rangle_t =\int_{t_0}^{t\wedge T} |\beta^\dagger_s f(X_s)|^2 \,\rd s \le C_1\int_{t_0}^{t\wedge T} |f(X_s)|^2\,\rd s.
		$$
		On the other hand,
		$$
		\int_{t_0}^ {t\wedge T} \langle f(X_s) ,\alpha_s\rangle \,\rd s \le - C_2 \int_{t_0}^ {t\wedge T} | f(X_s)|^2 \,\rd s.
		$$
		Consequently,
		$$
		\int_{t_0}^ {t\wedge T} \langle f(X_s) ,\alpha_s\rangle \,\rd s+M_t\le M_t- \frac {C_2}{C_1} \langle M\rangle_t
		$$
		We conclude that for $t\ge t_0$
		\begin{align*}
		\underbrace{F(X_{t\wedge T})}_{\ge -C_1} -F(X_{t_0}) &= \underbrace{
			\int_{t_0}^ {t\wedge T} \langle f(X_s) ,\alpha_s\rangle \,\rd s+M_t}_{\le M_t- \frac {C_2}{C_1} \langle M\rangle_t} +  \frac 12\underbrace{ \int_{t_0}^{t\wedge T} \tr(\beta_{s}^\dagger H(X_{s}) \beta_{s}) \,\rd s}_{\le C_1}.
		\end{align*}
		Note that on $\{\langle M\rangle _\infty=\infty\}$, almost surely, $M_t-\frac {C_2}{C_1}\langle M\rangle_t\to -\infty$. Since the left-hand side is  bounded from below we get that $\langle M\rangle _\infty<\infty$, a.s., so that the martingale $(M_t)_{t \ge t_0}$ converges almost surely to a finite value. Consequently,
		$$
		-\int_{t_0}^ {t\wedge T} \langle f(X_s) ,\alpha_s\rangle \,\rd s\le  F(X_{t_0})+M_t+ \frac 32 C_1
		$$
		and the left-hand side is increasing in $t$ with a finite limit.   Consequently, on $$
			\Big\{ \Big(\int_0^t \tr(\beta_s^\dagger H(X_s)\beta_s) \, \rd s\Big)_{t \ge 0} \text{ converges}\Big\},
			$$
			we have that $(F(X_{t\wedge T}))_{t\ge 0}$ converges and $-\int_{t_0}^ {T} \langle f(X_s) ,\alpha_s\rangle \,\rd s<\infty$, almost surely. We note that 
			$$
			\IC_1=\Big\{ \Big(\int_0^t \tr(\beta_s^\dagger H(X_s)\beta_s) \, \rd s\Big)_{t \ge 0} \text{ converges}\Big\} \cap \bigcup_{n\in\N} \{T^{(n,n,1/n)}=\infty\}
			$$
			which implies the first statement.
		%

		To show the  second statement we set
		$$
		\tilde T:=\tilde T^{(t_0,C_1,C_2)}:= T^{(t_0,C_1,C_2)}\wedge \inf\Bigl\{ t\ge t_0: |\alpha_t|\vee  \sup_{y\in B(X_t,C_2)}|H(y)|\vee \int_{t_0}^t |\beta_s|^2_F\,\rd s >C_1\Bigr\}
		$$  
		and note that it satisfies to show that $f(X_t) \to 0$, almost surely, on $\{\tilde T^{(t_0,C_1,C_2)}=\infty\}$, since
		$$
			\IC_2= \bigcup_{n\in\N} \{\tilde T^{(n,n,1/n)}=\infty\}.
		$$
		Consider
		$$
		(\tilde M_t)_{t \ge t_0}:= \Bigl( \int_{t_0}^{t\wedge \tilde T} \beta_s \, \rd W_s \Bigr)_{t\ge t_0} \text{ \ and \ }(B_t)_{t\ge t_0}:=\Bigl( \int_{t_0}^{t\wedge \tilde T} \langle f(X_s), \alpha_s\rangle \,\rd s \Bigr)_{t\ge t_0}
		$$
		Note that $(\tilde M_t)_{t \ge t_0}$ is a multivariate  martingale with
		$$
		\langle \tilde M \rangle_t =  \int_0^{t\wedge \tilde T} |\beta_s|_{F}^2 \, \rd s \le C_1.
		$$
		Hence, $(\tilde M_t)_{t \ge t_0}$ converges almost surely and together with the proof of the first statement we get that
		$$
		\Omega_0 := \bigl\{(\tilde M_t)_{t \ge t_0} \text{ and } (B_t)_{t \ge t_0} \text{ converge}\bigr\}
		$$
		is an almost sure event. Suppose now that there exists $\omega \in \Omega_0 \cap  \{ \tilde T=\infty\}$ for which $(f(X_t(\omega)))_{t \ge 0}$ does not converge to zero. Then there exist $\delta>0$ and an increasing $[t_0,\infty)$-valued sequence $(t_k)_{k \in \N}$ converging to infinity with 
		$$
		|f(X_{t_k}(\omega))|>\delta
		$$
		for all $k \in \N$. 
		Pick $\eps\in(0,C_2)$ with $C_1 \eps\le \delta/2$.
		By thinning the original sequence $(t_k)_{k \in \N}$ we can ensure that $([t_k,t_k+\eps/(2C_1 )]:k \in \N)$ are disjoint intervals and that   for all $k \in \N$
		$$
		\sup\limits_{t\ge t_k}|\tilde M_t(\omega)-\tilde M_{t_k}(\omega)|<\eps/2.
		$$
		For $k\in\N$ and  $t\in[t_k,t_k+\eps/(2C_1 )]$ one has
		\begin{align*}
		|X_t(\omega)-X_{t_k}(\omega)|&\le \int_{t_k}^t |\alpha_s(\omega)| \, \rd s + |\tilde M_t(\omega)-\tilde M_{t_k}(\omega)| \\
		& \le C_1  (t-t_k) + \eps/2 \le \eps
		\end{align*}
		Moreover, since $\eps<C_2$ we conclude for $x=X_t(\omega)$ and $y=X_{t_k}(\omega)$ that
		{
			\begin{align*}
			|f(y)-f(x)| \le |H(\xi)|\, |y-x|\le C_1 |y-x|,
			\end{align*}
		}
		where $\xi$ lies on the segment joining $x$ and $y$.
		Hence,
		$$
		|f(X_t(\omega))|\ge |f(X_{t_k}(\omega))|- \underbrace{|f(X_t)-f(X_{t_k}(\omega))|}_{\le \delta/2} \ge \delta/2,
		$$
		so that
		$$
		\int_{t_k}^{t_k+\eps/(2C_1)} |f(X_t(\omega))|^2 \, \rd t \ge   \Bigl( \frac \delta 2 \Bigr)^2 \frac{\eps}{2 C_1} .
		$$
		This entails that $\int_{t_0}^\infty |f(X_t)|^2\,\rd t=\infty$ and, hence,
		$$
		-B_t(\omega)=\int_{t_0}^{t} \langle f(X_s(\omega)),-\alpha_s(\omega)\rangle\,\rd s\ge C_2 \int_{t_0}^{t} |f(X_s(\omega)|^2\,\rd s \to \infty.
		$$
		This contradicts that $(B_t(\omega))_{t\ge t_0}$ converges to a finite value, by choice of $\omega$.

		It remains to prove almost sure convergence of $(X_t)_{t \ge 0}$ on $\IC_3$ in the case where the set $\cC:=\{x \in \R^d:f(x)=0\}$ is at most countable.
		Take $\omega \in \Omega$ for which  $\lim_{t\to\infty} f(X_t(\omega)) =0$ and set $r:=\limsup_{t \to \infty} |X_t(\omega)|<\infty$. We show by contradiction that $(X_t(\omega))_{t \ge 0}$ converges. Suppose that $(X_t(\omega))_{t \ge 0}$ would not converge. Then there would exist a coordinate $i \in\{1,\dots,d\}$ and reals $a<b$ such that the set of accumulation points 
		$$
		\cK:=\bigcap_{n\in\N} \bigl\{X^{(i)}_t(\omega): t\ge n,  X_t(\omega)\in \overline{B(0,r+1)}\bigr\}
		$$
		is a compact set containing $[a,b]$. By applying the principle of nested intervals we can produce for every $u\in [a,b]$ an accumulation point $\eta(u)$ of $(X_t(\omega))_{t \ge 0}$ with $i$th coordinate equal to $u$. Recall that $f(X_t)\to0$ so that $\eta(u)$ is a critical point of $F$. We thus observe that there is an injective mapping taking $[a,b]$ to the set of critical points of $F$ which contradicts the assumption.
	\end{proof}
	\black
	\section{Proof of Theorem~\ref{theo2}}\label{sec3}
	The proof of Theorem~\ref{theo2} is arranged as follows. In Section~\ref{sec31}, we provide moment estimates for the Lyapunov function seen by the SDE in a neighbourhood $U$ of a point $y \in \R^d$ so that the \CL ojasiewicz-inequality is satisfied on $U$ with parameters $\CL>0$ and $\theta \in (\frac 12, 1)$. To that purpose we fix a particular critical level and ``take out'' realisations that undershoot the critical level by a certain amount (\emph{lower dropout}).  In the case where no perturbation is present ($\beta_t \equiv 0$), we recover the convergence rate for the respective ODE (\ref{eq:ODE}) derived in \cite{lojasiewicz1984sur}, i.e.
	$$
	|F(x_t)-F(y)| \le C  (t+1)^{-\frac{1}{2\theta-1}}.
	$$
	We show that the convergence rate is maintained in the SDE setting as long as $|\beta_t|\le \tilde C  (t+1)^{-\frac{\theta}{2\theta-1}}$ for sufficiently large $t$. In Section~\ref{sec32}, we provide estimates for   the drift term $(\int_0^t \alpha_s \, \rd s)_{t \ge 0}$ in the setting of Section~\ref{sec31} on the event that no dropout occurs.
	In Section~\ref{sec33}, we provide estimates that later allow us to show that when a dropout occurs for a particular critical level it is very likely that the limit of the Lyapunov function lies strictly below that critical level. Finally, in Section~\ref{sec36} we combine the results of the previous subsections to achieve the proof of Theorem~\ref{theo2}.
	
	\subsection{Convergence rate for the target value in the case without lower dropout} \label{sec31}
	
	\begin{prop}\label{prop:loja1}Suppose that the following assumptions are satisfied:\smallskip
		
		\noindent{\bf Assumptions on the error term.}
		Let  $\theta\in(\frac 12,1)$, $(v_t)_{t \ge 0}$ be a continuously differentiable, decaying, positive function {and $(\sigma_t)_{t \ge 0}$ be a right-continuous and locally integrable function}, so that there exist constants $\kappa_1, \kappa_2 >0$ such that for  all $t \ge 0$
		\begin{align} \label{eq:4}
		-\frac{\dot v_t}{v_t} \le \kappa_1 v_t^{2\theta-1} \quad \text{ and } \quad \kappa_2 v_t^{2\theta} \ge  \sigma^2_t.
		\end{align}
		\noindent{\bf {No lower dropout/\loja-inequality.}}
		Let $\ell \in \R$, $\CL>0$ and $(w_t)_{t \ge 0}=(C_w v_t)_{t \ge 0}$ for a $C_w\ge 0$. Let $T$ be a stopping time, such that for all $t \ge 0$, on $\{T>t\}$,
		\begin{align*}
		|f(X_t)|\ge \CL |F(X_t)-\ell|^\theta\  , \ F(X_t)-\ell\ge -w_t 
		\end{align*}
		as well as
		\begin{align*}
		-\langle f(X_t),\alpha_t\rangle \ge \rho |f(X_t)|^2 , \ \ 
		\sfrac12 \tr(\beta_t^\dagger H(X_t) \beta_t)\le \sigma^2_t
		\end{align*}
		for a $\rho>0$.  
		Additionally, suppose that for every $t>0$
		\begin{align}\label{eq7246}
		\E\Bigl[\int_0^{t\wedge T} |\beta_s^\dagger f(X_s)|^2\, \rd s\Bigr]<\infty.
		\end{align}

		\noindent	{\bf Result.} Suppose that the above assumptions hold and
		$$
		\E[\1_{\{T>0\}}  (F(X_0)-\ell)]\le R.
		$$
		Then there exist $\alpha,\eta >0$ such that, for all $t\ge 0$, one has
		\begin{align} \label{eq:4531}
			\E\bigl[\1_{\{T>t\}} (F(X_t)-\ell+w_t) \bigr]\le \Phi^{(R)}_{t}  +\alpha v_t,
		\end{align} 
		where 
		$$
		\Phi^{(R)}_t=R\Bigl(\big((2\theta-1)\eta R^{2\theta-1}\big)t+1\Bigr)^{-1/(2\theta-1)}.
		$$
%
	\end{prop}

	\begin{proof}
		We assume without loss of generality that $\ell=0$. By Itô's formula
		\begin{align} \label{eq:itoF}
		\rd F(X_t)= \langle f(X_t) ,\alpha_t\rangle \,\rd t+   \langle f(X_t),\beta_t\,\rd W_t\rangle +\frac 12 \tr(\beta_t^\dagger H(X_t) \beta_t)\,\rd t.
		\end{align}
		Let $(Z_t)_{t\ge 0}:= (\1_{\{T>t\}} (F(X_t)+w_t))_{t \ge 0}$. By choice of $T$, $(Z_t)_{t \ge 0}$ is a non-negative process
		and we have
		\begin{align} \label{eq:7}
		\rd Z_t =  \1_{\{T > t\}} \bigl(\langle f(X_t) ,\alpha_t\rangle  +\frac 12 \tr(\beta_t^\dagger H(X_t) \beta_t)+\dot w_t \bigr)\,\rd t+\rd M_t- \rd \xi_t,
		\end{align}
		where $(M_t)_{t \ge 0}$ denotes the $L^2$-martingale 	 
		$$
		(M_t)_{t \ge 0}:=\Bigl( \int_0^{T\wedge t}  \langle f(X_s)),\beta_s\,\rd W_s\rangle \Bigr)_{t\ge 0},
		$$
		see assumption~(\ref{eq7246}), and  $(\xi_t)_{t \ge 0}$ is an increasing process given by
		$$
		\xi_t:=\begin{cases}
		0, & \text{ if }t<T\text{ or }T=0,\\
		F(X_T)+w_T, & \text{ else}.
		\end{cases} 
		$$
		For all $t\ge 0$, we let $\zeta_t:=\E[Z_t]$. In the first step we show that $(\zeta_t)_{t \ge 0}$ is right-continuous and has no upward jumps. By definition, $(Z_t)_{t \ge 0}$ is right-continuous and it is straight-forward to verify the continuity theorem for integration (Note that $Z_t$ is non-negative and $\langle f(X_t),\alpha_t\rangle \le 0$ on $\{t<T\}$). To verify that $(\zeta_t)_{t \ge 0}$ has no upward jump at $t>0$ we note that for $\eps\in(0,t)$
		$$
		\zeta_t-\zeta_{t-\eps}= \E[Z_t-Z_{t-\eps}]\le \E\Bigl[\int_{t-\eps}^t \1_{\{s<T\}} \frac 12 \tr(\beta_s^\dagger H(X_s) \beta_s)\,\rd s\Bigr]\to0
		$$
		as $\eps\downarrow 0$.
		
		Take $\alpha\ge C_w$ and $\eta>0$ such that for all $a,b>0$
		\begin{align} \label{eq:ineq:1}
			\rho\L^22^{-(2\theta-1)} (a+b)^{2\theta} > \eta \,a^{2\theta}+\frac{\rho\L^2 C_w^{2\theta}+\kappa_2+\kappa_1 \alpha}{\alpha^{2\theta}} \,b^{2\theta}.
		\end{align}
		We will now show that $\zeta_t\le \bar\zeta_t:=\Phi^{(R)}_t +\alpha v_t$ for all $t\ge 0$. Note that $\zeta_0\le R+C_w v_0\le \Phi_0^{(R)}+\alpha v_0=\bar \zeta_0$ since $\alpha\ge C_w$ so that the inequality is at least true for $t=0$. We prove the statement by contradiction.
		Suppose the time
		$$
		\tau:=\inf\{s\ge 0: \zeta_s>\bar \zeta_s\}
		$$
		is finite. Since $\zeta-\bar \zeta$ has no upwards jumps and starts in a non-positive value we have $\zeta_\tau=\bar\zeta_\tau$.
		We consider the slope of the right secant of $\zeta$ with supporting points $\tau$ and $\tau+\eps$ for $\eps >0$. 
		Using the \loja-inequality we get that  on $\{T >t\}$
		\begin{align*}
		\langle f(X_t),\alpha_t \rangle &\le -\rho |f(X_t)|^2 \le - \rho \CL^2 |F(X_t)|^{2\theta} \\
		& \le -  \rho\CL^2 \bigl(2^{-(2\theta-1)}|F(X_t)+w_t|^{2\theta}-w_t^{2\theta} \bigr)
		\end{align*}
		so that with~(\ref{eq:7}) and $\frac 12\tr(\beta^\dagger_t H(X_t)\beta_t)\le \sigma^2(t)$
		$$
		\frac{	\zeta_{\tau+\eps}- \zeta_\tau}\eps\le  \E\Bigl[\frac 1\eps\int_{\tau}^{\tau+\eps }\1_{\{T>t\}}\bigl(-  \rho\CL^2 \bigl(2^{-(2\theta-1)} |Z_t|^{2\theta}-w_t^{2\theta} \bigr)
		+ \sigma_t^2\bigr)\,\rd t\Bigr].
		$$
		By non-negativity and right-continuity of $(Z_t)_{t \ge 0}$ and convexity of $|\cdot|^{2\theta}$ one has
		$$
		\liminf_{\eps\downarrow0} \E\Bigl[\frac 1\eps \int_{\tau }^{\tau+\eps} \1_{\{T>t\}} |Z_t|^{2\theta}\,\rd t\Bigr]\ge \E[|Z_\tau|^{2\theta}]\ge |\E[Z_\tau]|^{2\theta}=\bar \zeta_\tau^{2\theta}.
		$$
		Consequently,
		\begin{align} \label{eq:Revision1}
		\limsup_{\eps\downarrow0} \frac {\zeta_{\tau+\eps}-\zeta_\tau}\eps\le -\rho \L^2(2^{-(2\theta-1)} \bar \zeta_\tau^{2\theta} -w_\tau^{2\theta})+ \sigma_\tau^2
		\end{align}
		Conversely, $(\Phi^{(R)}_t)_{t\ge 0}$ is  the solution of the differential equation
		\begin{align*}
			\dot	\Phi^{(R)}_t=-\eta (\Phi^{(R)}_t)^{2\theta}, \Phi^{(R)}_0=R
		\end{align*} 
		so that
		$$
		\dot{\bar \zeta}_\tau= -\eta (\Phi^{(R)}_\tau)^{2\theta}+\alpha \dot v_\tau.
		$$
		Note that if the  right-hand side of (\ref{eq:Revision1}) is strictly smaller than $\dot{\bar \zeta}_\tau$, i.e., if
		\begin{align}\label{eq845}
		-\rho \L^2(2^{-(2\theta-1)} \bar \zeta_\tau^{2\theta} -w_\tau^{2\theta})+ \sigma_\tau^2<-\eta (\Phi^{(R)}_\tau)^{2\theta}+\alpha \dot v_\tau,
		\end{align}
		then there exists an interval $[\tau,\tau+\eps]$ on which $\zeta_t\le \bar \zeta_t$ which entails that $\zeta$ cannot overtake $\bar \zeta$ at time $\tau$ which is a contradiction to the choice of $\tau$. Using that $w_\tau=C_w v_\tau$, $\dot v_\tau\ge -\kappa_1v_\tau^{2\theta}$ and $\sigma_\tau^2\le \kappa_2 v_\tau^{2\theta}$ we conclude that (\ref{eq845}) is satisfied, if
		$$
		\rho\L^22^{-(2\theta-1)}(\underbrace{\Phi_\tau^{(R)}+\alpha v_\tau}_{=\bar \zeta_\tau})^{2\theta}>\eta (\Phi_\tau^{(R)} )^{2\theta} +\frac{\rho\L^2 C_w^{2\theta}+\kappa_2+\kappa_1\alpha}{\alpha^{2\theta}} (\alpha v_\tau)^{2\theta}.
		$$
		By choice of $\alpha$ and $\eta$ the latter inequality holds so that we produced a contradiction. This finishes the proof.
	\end{proof}

	\begin{rem} \label{rem:loja1} 
	We discuss Proposition~\ref{prop:loja1} in the case where the  critical level  $\ell$  is a local minimum and  choose $C_w=0$ in the proposition.  
		For given $\theta\in(\frac12,1)$, we note that for $v_t=(t+1)^{-1/(2\theta-1)}$  the condition on the left-hand side of~(\ref{eq:4}) is satisfied for an appropriate constant~$\kappa_1$ and if $\sigma_t=\cO((t+1)^{-\theta/(2\theta-1)})$, then also  the second condition is satisfied for an appropriate  $\kappa_2$. In this case,
$(v_t)_{t\ge0}$ is of the same order as $(\Phi_t^{(R)})_{t\ge0}$ so that
$$
\E[\1_{\{T>t\}} (F(X_t)-\ell)]=\cO\bigl((t+1)^{-1/(2\theta-1)}\bigr).
$$
This agrees with the corresponding estimates for the ODE convergence. So if $(\sigma_t)_{t \ge 0}$ is of order $ \cO((t+1)^{-\theta/(2\theta-1)}$ we essentially get the same order of convergence as in the deterministic setting.

Now suppose that $\sigma_t\approx (t+1)^{-\sigma}$ with $ 0\le \sigma<\theta/(2\theta-1)$. Then one can choose $v_t=(t+1)^{-\sigma/\theta}$ and constants  $\kappa_1$ and $\kappa_2$ appropriately so that~(\ref{eq:4}) is satisfied. 
So in that case
$$
\E[\1_{\{T>t\}} (F(X_t)-\ell)]=\cO\bigl((t+1)^{-\sigma/\theta}\bigr).
$$
Altogether, we thus get for the choice $\sigma_t\approx (t+1)^{-\sigma}$ with $ \sigma \in [0,\infty)$ that
$$
\E[\1_{\{T>t\}} (F(X_t)-\ell)]=\cO\bigl((t+1)^{-(\frac \sigma\theta\wedge \frac 1{2\theta-1})}\bigr).
$$
Observing that on a compact set $F$, $f$ and the Hessian are uniformly bounded it is straight-forward to infer the statement of Theorem~\ref{the:rate}.
	\end{rem}
	
	
	\subsection{Bounding the drift term in the case where no dropout occurs}\label{sec32}
	
	\begin{prop} \label{prop83256}
		Suppose that $T$ is such that on $\{T\ge t\}$, $\langle f(X_t),-\alpha_t\rangle\ge 0$ and $(F(X_t))_{t \ge 0}$ converges on $\{T=\infty\}$, almost surely. Let $\Phi:[0,\infty)\to[0,\infty)$ be a decreasing, continuously differentiable function so that $t\mapsto -\dot \Phi_t$ is decreasing and
		$$
		\lim_{t\to\infty} \frac {\Phi_t}{\sqrt{-\dot \Phi_t}}=0.
		$$
		If for every $t\ge0$,
		$$
		\E[\1_{\{T> t\}} (F(X_t)-F(X_T))]+\frac12 \E\Bigl[\int_t^\infty \1_{\{s\le T\}} \tr(\beta_s^\dagger H(X_s) \beta_s)\,\rd  s\Bigr]\le \Phi_t
		$$
		and  $(M_t)_{t \ge 0}$ given by
		$$
		M_t:=\int_0^{t\wedge T} \langle f(X_s), \beta_s\,\rd W_s\rangle
		$$
		is a regular martingale, then
		$$
		\E\Bigl[\int_{0}^\infty \1_{\{s<T\}} \sqrt{\langle f(X_s),-\alpha_s\rangle}\,\rd s\Bigr]\le \int_0^\infty \sqrt{-\dot \Phi_s}\,\rd s.
		$$
	\end{prop}
	
	\begin{proof}Since $(M_t)_{t \ge 0}$ is a regular martingale we get with It\^o's formula, see  (\ref{eq:itoF}), that
		\begin{align*}
		\Psi(t):=&\ \E\Bigl[\int_{t}^\infty \1_{\{s<T\}}\langle f(X_s),-\alpha_s\rangle\,\rd s\Bigr]\\
		=&\  \E\bigl[\1_{\{T>t\}} (F(X_{t})-F(X_T))\bigr]+\frac 12 \E\Bigl[\int_t^\infty \1_{\{T>s\}} \tr(\beta_s^\dagger H(X_s) \beta_s)\,\rd  s\Bigr]\le \Phi_t. 
		\end{align*}
		The function $\Psi$ is right-continuous and monotonically decreasing and, hence, c\`adl\`ag.
		We let $\varphi_s:=(-\dot\Phi_s)^{-1/2}$ and note that by Cauchy-Schwarz
		\begin{align*}
		\E\Bigl[\int_{0}^\infty \1_{\{s<T\}} \sqrt{\langle f(X_s),-\alpha_s\rangle}\,\rd s\Bigr]\le \Bigl( \int_{0}^\infty \varphi_s \E\bigl[\1_{\{s<T\}}  \langle f(X_s),-\alpha_s\rangle\bigr]\,\rd s\Bigr)^{1/2}\Bigl(\int_0^\infty \varphi_s^{-1}\,\rd s\Bigr)^{1/2}.
		\end{align*}
		Moreover, using partial integration we get that
		\begin{align*}
		\int_{0}^\infty \varphi_s \E\bigl[\1_{\{s<T\}}  \langle f(X_s),-\alpha_s\rangle\bigr]\,\rd s&=-\int_{0}^\infty \varphi_s\,\rd \Psi_s\\
		&= {-}[\varphi \Psi]_0^\infty+\int_0^\infty \Psi_s\,\rd \varphi_s.
		\end{align*}
		$(\varphi_s)_{s \ge 0}$ is increasing by assumption and $\varphi_s \Phi_s\to0$ as $s\to \infty$ so that 
		\begin{align*}
		{ -}[\varphi \Psi]_0^\infty+\int_0^\infty \Psi_s\,\rd \varphi_s&\le {-}\varphi_0 (\Phi_0-\Psi_0){-}[\varphi\Phi]_0^\infty +\int_0^\infty \Phi_s\,\rd \varphi_s\\
		&={-}\varphi_0 (\Phi_0-\Psi_0)+\int_{0}^\infty \underbrace{\varphi_s (-\dot\Phi_s)}_{= \sqrt{-\dot \Phi_s}}\,\rd s.
		\end{align*}
		Altogether we get that
		$$
		\E\Bigl[\int_{0}^\infty \1_{\{s<T\}} \sqrt{\langle f(X_s),-\alpha_s\rangle}\,\rd s\Bigr]\le\int_0^\infty \sqrt{-\dot \Phi_s}\,\rd s.
		$$
	\end{proof}

	\subsection{Technical analysis of lower dropouts}\label{sec33}
	
	Roughly speaking, the following two lemmas will later be used to show that for a certain critical level of the target function, a lower dropout (in the sense of the previous two propositions) entails that the Lyapunov function converges to a value strictly below the respective critical level with high probability.

	\begin{lemma}\label{le:mart}
		Let $(M_t)_{t \ge 0}$ be a continuous local martingale started in zero. Then for every $\kappa>0$
		$$
		\P\Bigl(\sup_{t\ge 0} (M_t-\langle M\rangle_t) \ge \kappa\Bigr)\le \frac 1{\kappa^2}+\sum_{n\in\N_0} \frac{2^{n+1}}{(2^n+\kappa)^2}=:\phi(\kappa).
		$$
		In particular, for every $\eps>0$ there exists a $\kappa>0$ such that the above right-hand side is smaller than $\eps$. 
	\end{lemma}
	
	\begin{proof}
		For $n \in \N_0$, let $T_n:=\inf\{t\ge 0: \langle M\rangle _{t}>2^n\}$. Then we have
		\begin{align*}
		\sup_{t\in [T_n,T_{n+1})} (M_t-\langle M\rangle_t)\le \sup_{t \in [0,T_{n+1})} M_t -2^n
		\end{align*}
		and
		$$
		\sup_{t \in[0,T_0)} (M_t-\langle M\rangle_t)\le \sup_{t\in [0,T_{0})} M_t.
		$$ 
		We use Doob's $L^2$-inequality to deduce that
		$$
		\P\Bigl(\sup_{t\in [0,T_{n+1})} M_t \ge 2^{n}+\kappa\Bigr)\le  (2^{n}+\kappa)^{-2} \E[M_{T_{n+1}}^2]\le    \frac{2^{n+1}}{(2^{n}+\kappa)^{2}}
		$$
		and
		$$
		\P\Bigl(\sup_{t\in [0,T_{0})} M_t\ge \kappa\Bigr)\le \frac 1{\kappa^2}.
		$$
		Therefore,
		$$
		\P\Bigl(\sup_{t\ge 0} (M_t-\langle M\rangle_t) \ge \kappa\Bigr)\le \frac 1 {\kappa^2}+\sum_{n\in\N_0} \frac{2^{n+1}}{(2^n+\kappa)^2}.
		$$
	\end{proof}

	\begin{lemma}\label{le:47236}
		Let $t_0\ge0$, $\rho,\kappa:[t_0,\infty)\to(0,\infty)$ be functions  and $t_0\le T'\le T$ be two stopping times such that,  for every $t\ge t_0$, on $\{T'\le t\le T\}$,
		$$|\beta_t^\dagger f(X_t)|^2\le \rho_{T'} \langle f(X_t), -\alpha_t\rangle \text{ \ 
		and \ }
		\frac 12 \int_{T'}^{t} \tr(\beta_s^\dagger H(X_s)\beta_s)\,\rd s\le \kappa_{T'}.
		$$
		Then,
		\begin{align*}
		\P\Bigl(\sup_{t\in[T',T]}& F(X_t)-F(X_{T'}) \ge 2\kappa_{T'}\Bigr) \le \phi\Bigl(\inf\limits_{t\ge t_0} \frac {\kappa_{t}}{\rho_{t}}\Bigr).
		\end{align*}
	\end{lemma}
	
	\begin{proof}
		We use again  representation (\ref{eq:1}):
		\begin{align*}
		\rd F(X_t)&= \langle f(X_t) ,\alpha_t\rangle \,\rd t+ \langle f(X_t),\beta_t\,\rd W_t\rangle +\frac 12 \tr(\beta_t^\dagger H(X_t) \beta_t)\,\rd t.
		\end{align*}
		Consider the local martingale
		$$
		(M_t)_{t\ge t_0}:=\Bigl( \int_0^t \1_{\{T'\le s\le T\}} \langle f(X_s),\beta_s\,\rd W_s\rangle \Bigr)_{t \ge t_0}
		$$
		and note that
		$$
		\langle M\rangle_t=\int_0^t \1_{\{T'\le s\le T\}} |\beta_s^\dagger f(X_s)|^2\,\rd s\le \rho_{T'} \int_0^t \1_{\{T'\le s\le T\}} \langle f(X_s),-\alpha_s\rangle\,\rd s.
		$$
		Hence, 
		\begin{align*}
		\int_0^t \1_{\{T'\le s\le T\}}  \langle f(X_s),\alpha_s\rangle \, \rd s &+ \int_0^t \1_{\{T'\le s\le T\}} \langle f(X_s),\beta_s\,\rd W_s\rangle\\
		&\le M_t -\frac {1}{\rho_{T'}} \langle M\rangle_t =:\Xi_t.
		\end{align*}
		Note that $\Xi_t=\rho_{T'}(\frac {1}{\rho_{T'}} M_t-\langle \frac{1}{\rho_{T'}} M\rangle_t)$ with $(\frac {1}{\rho_{T'}} M_t)_{t \ge t_0}$ being a continuous local martingale started in zero so that with Lemma~\ref{le:mart} 
		\begin{align*}
		\P\Bigl(\sup_{t\ge t_0} \Xi_t \ge \kappa_{T'}\Bigr) &= \P\Bigl(\sup_{t\ge t_0}\frac {1}{\rho_{T'}} M_t-\langle \frac{1}{\rho_{T'}} M\rangle_t
		\ge \frac{\kappa_{T'}}{\rho_{T'}}\Bigr)\\
		&\le \P\Bigl(\sup_{t\ge t_0}\frac {1}{\rho_{T'}} M_t-\langle \frac{1}{\rho_{T'}} M\rangle_t
		\ge \inf\limits_{t\ge t_0} \frac {\kappa_{t}}{\rho_{t}}\Bigr) \le \phi\Bigl(\inf\limits_{t\ge t_0} \frac {\kappa_{t}}{\rho_{t}}\Bigr).
		\end{align*}
		Altogether, using
		\begin{align*}
			\sup_{T'\le t\le T} F(X_t)-F(X_{T'}) \le \sup_{T'\le t\le T} \Xi_t + \underbrace{\sup_{T'\le t\le T} \frac 12 \int_{T'}^{t} \tr(\beta_s^\dagger H(X_s)\beta_s)\,\rd s}_{\le \kappa_{T'}}
		\end{align*}
		we get that
		\begin{align*}
		\P\Bigl(\sup_{T'\le t\le T} F(X_t)-F(X_{T'}) \ge 2\kappa_{T'}\Bigr) &\le 
		\P\Bigl(\sup_{t \ge t_0} \Xi_t \ge \kappa_{T'} \Bigr) \le \phi\Bigl(\inf\limits_{t\ge t_0} \frac {\kappa_{t}}{\rho_{t}}\Bigr).
		\end{align*}
	\end{proof}

	\subsection{Proof of Theorem~\ref{theo2}}\label{sec36}
	
	We cite Lemma~3.7 from \cite{dereich2021convergence}.
	
	\begin{lemma}\label{lem:lojaset}Let $F:\R^d\to\R$ be a \loja-function with differential $f$ and $K\subset \R^d$ be an arbitrary compact set. 
		\begin{enumerate}
			\item The set of \emph{critical levels}
			$$
			\cL_K:=\{F(x):x\in \cC_F\cap K\}
			$$
			is finite so that $F$ has at most a countable number of critical levels.
			\item For every critical level $\ell\in \cL_K$ there exists an open neighbourhood $U\supset  F^{-1}(\{\ell\})\cap K$, $\CL>0$, $\theta\in[\frac 12 ,1)$ such that for every $y\in U$
			$$
			|f(y)|\ge \CL |F(y)-\ell|^\theta.
			$$
			\item For a neighbourhood as in (2), there exists $\eps>0$ such that
			$$ F^{-1}((\ell-\eps,\ell+\eps))\cap K\subset U.
			$$
		\end{enumerate}
	\end{lemma}
	
	Now we are able to  prove the main result of this article.
	
	\begin{proof}[Proof of Theorem~\ref{theo2}]
		For $\rho,C, t_0 >0$ and a  set $K\subset \R^d$ let 
		\begin{align*}
		T_{\rho,C,t_0,K}:= \inf\Bigl\{ t\ge t_0: \ & {\langle f(X_t),-\alpha_t\rangle}< \rho |f(X_t)|^2  , \   |\alpha_t|> C|f(X_t)|, \ |f(X_t)|>C,  \\
		& 
		 |\beta_t|^2> C\sigma_t^2, \  \frac12 \mathrm {tr}(\beta_t^\dagger H(X_t)\beta_t)> C \sigma_t^2 \ \text{ or } \ X_t \notin K  \Bigr\}.
		\end{align*}
		Set $\theta_0:=\frac \sigma{2\sigma-1}\in(\frac12,1)$.
		
		1) First let $U$ be an open and bounded set such that for a $\theta\in (\theta_0,1)$, an $\ell\in\R$ and $\L>0$, for all $y\in U$
		$$
		|f(y)|\ge \L \,|F(y)-\ell|^\theta.
		$$
		
		1.a) We show that on $\{T_{\rho,C,t_0,U}=\infty\}$, almost surely, the random set
		$$
		S:=\{s\ge 0: F(X_s)<\ell- (s+1)^{-\frac 1{2\theta-1}}\}
		$$
		is bounded.

		By Proposition~\ref{prop:conv1} we have on $\{T_{\rho,C,t_0,U}=\infty\}$ that, almost surely, $(F(X_t))_{t \ge 0}$ converges and $f(X_t) \to 0$, which together with the \loja-inequality implies that $F(X_t) \to \ell$.

	 By choice of $\theta$ one has  $2\sigma > \frac{2\theta}{2\theta-1}$. Thus, for $t_1\ge t_0$ large enough we get for all $t \ge t_1$
	 $$
	 C\sigma_t^2 \le \frac{1}{4\theta-2} (t+1)^{-\frac{2\theta}{2\theta-1}}.
	 $$
	 Let $T'_{t_1}:= \inf\{t\ge t_1: F(X_t)<\ell- (t+1)^{-\frac 1{2\theta-1}}\}$, then
	 one has 
		$$
			\sup_{T'_{t_1}\le t\le T_{\rho,C,t_0,U}} \int_{T'_{t_1}}^t \frac 12 \tr(\beta_s^\dagger H(X_s)\beta_s)\,\rd s\le \frac{1}{4\theta-2} \int_{T'_{t_1}}^\infty (s+1)^{-\frac{2\theta}{2\theta-1}} \, \rd s = \frac{1}{2} (T'_{t_1}+1)^{-\frac{1}{2\theta-1}}.
		$$
		Furthermore,   we have for $t$ with  $T'_{t_1}\le t< T_{\rho,C,t_0,U}$
		$$
			|\beta_t^\dagger f(X_t)|^2\le |\beta_t^\dagger|^2 |f(X_t)|^2 \le \frac{C}{\rho} (T'_{t_1}+1)^{-2\sigma} \langle f(X_t), -\alpha_t\rangle
		$$
		so that, using Lemma~\ref{le:47236}, we get 
		\begin{align*}
			\P(S\cap [t_1,\infty)\not= \emptyset , T_{\rho,C,t_0,U} =\infty)&= \P(T'_{t_1}<\infty, T_{\rho,C,t_0,U}=\infty, F(X_t)\to \ell) \\
			&\le \P\Bigl(\sup_{t\in[T'_{t_1},T_{\rho,C,t_0,U})} F(X_t)-F(X_{T'_{t_1}}) \ge (T'_{t_1}+1)^{-\frac{1}{2\theta-1}}\Bigr)\\
			& \le \phi\Bigl(\frac {\rho  }{2 C}(t_1+1)^{2\sigma-\frac{1}{2\theta-1}}\Bigr) \overset{t_1 \to \infty}{\longrightarrow} 0
		\end{align*}
		and for all $t_1\ge t_0$
 		$$
			\P(S \text{ is unbounded}, T_{\rho,C,t_0,U}=\infty)\le \P(S\cap [t_1,\infty)\not= \emptyset , T_{\rho,C,t_0,U} =\infty) \to 0.
		$$

		1.b) Let 
		$$
		\bar T_{\rho,C,t_0,U}:= T_{\rho,C,t_0,U} \wedge \inf\{t\ge t_0: X_t<\ell-(t+1)^{-\frac 1{2\theta-1}} \}.
		$$
		We show that on $\{\bar T_{\rho,C,t_0,U}=\infty\}$ the process $(X_t)_{t \ge 0}$ converges, almost surely.

			We set $\sigma'_t:=\sqrt C \sigma_t$ and $v_t:=w_t:= (t+1)^{-\frac{1}{2\theta-1}}$ (so that $C_w=1$) and $T := \bar T_{\rho,C,t_0,U}$ and show that the functions satisfy the Assumptions of Proposition~\ref{prop:loja1}, Proposition~\ref{prop83256} and the right-hand side of the estimate in Proposition~\ref{prop83256} is finite.
			
			First note that one can apply Proposition~\ref{prop:loja1} for the process $(X_t)_{t\ge t_0}$ started at time $t_0$ in $X_{t_0}$. In particular, we have 
			$$
			-\dot v_t= \frac{1}{2\theta-1}(t+1)^{-\frac{2\theta}{2\theta-1}}= \frac{1}{2\theta-1} v_t^{2\theta}
			$$
			and with $2\sigma > \frac{2\theta}{2\theta-1}$
			$$
			v_t^{2\theta}=  (t+1)^{-\frac{2\theta}{2\theta-1}} \ge \frac{1}{C}(\sigma'_t)^2
			$$
			so that (\ref{eq:4}) is satisfied. Further, with the definition of $T_{\rho,C,t_0,K}$
			$$
				\E\Bigl[\int_{t_0}^{T} |\beta_s^\dagger f(X_s)|^2\, \rd s\Bigr]\le \int_{t_0}^\infty C^3 \sigma_t^2 \, \rd t < \infty,
			$$
			so that (\ref{eq7246}) is satisfied. Now, we can choose $\eta>0$ sufficiently small and $\alpha>0$ sufficiently large so that (\ref{eq:ineq:1}) is satisfied and with the boundedness of $U$ and Proposition~\ref{prop:loja1} we get for sufficiently large $R, C'>0$ that for all $t\ge t_0$
			$$
				\E\bigl[\1_{\{T>t\}} (F(X_t)-\ell+w_t) \bigr]\le \Phi^{(R)}_{t}  +\alpha v_t \le C' (t-t_0+1)^{-\frac{1}{2\theta-1}}.
			$$
			Regarding Proposition~\ref{prop83256} we recall that, on $\{T=\infty\}$, we almost surely have $F(X_t) \to \ell$ so that $F(X_T)$ is well-defined. By the definition of $T$ we have $-F(X_T)\le -\ell +w_T$, where $w_\infty=0$, and the monotonicity of $(w_t)$ gives
			$$
				\E[\1_{\{T> t\}} (F(X_t)-F(X_T))]\le \E[\1_{\{T> t\}} (F(X_t)-\ell +w_t)] \le C' (t-t_0+1)^{-\frac{1}{2\theta-1}}.
			$$
			Thus, we have
			\begin{align*}
				\E[\1_{\{T> t\}} &(F(X_t)-F(X_T))]+\frac12 \E\Bigl[\int_t^\infty \1_{\{s\le T\}} \tr(\beta_s^\dagger H(X_s) \beta_s)\,\rd  s\Bigr] \\
				&\le C' (t-t_0+1)^{-\frac{1}{2\theta-1}} + C \int_{t}^\infty \sigma_s^2 \, \rd s \le \Phi_t,
			\end{align*}
			where $\Phi_t:= C''(t-t_0+1)^{-\frac{1}{2\theta-1}}$ for sufficiently large $C''$. Moreover, $\dot \Phi_t$ is an increasing function so that $t \mapsto -\dot \Phi_t$ is decreasing and using that $\theta<1$ we get that
			$$
			\Phi_t/\sqrt{-\dot \Phi_t} = \sqrt{2\theta-1}\,  (C'')^{\theta-1/2}\,  (\Phi_t)^{1-\theta} \to 0.
			$$
			Moreover, $(M_t)_{t \ge t_0} :=\bigl(\int_{t_0}^{t\wedge T} \langle f(X_s), \beta_s\,\rd W_s\rangle\bigr)_{t \ge t_0}$ is a regular martingale as
			$$
				\langle M_t \rangle_\infty = \int_{t_0}^T |\beta_s^\dagger f(X_s)|^2 \, \rd s \le \int_{t_0}^\infty C^3 \sigma_t^2 \, \rd t < \infty.
			$$
			Therefore, Proposition~\ref{prop83256} is applicable and we get
			\begin{align*}
				\E\Bigl[\int_{t_0}^\infty \1_{\{s<T\}} |\alpha_s|\, \rd s\Bigr]&\le C \, \E\Bigl[\int_{t_0}^\infty \1_{\{s<T\}} |f(X_s)|\, \rd s\Bigr]\\
				&\le C\sqrt \rho \, \E\Bigl[\int_{t_0}^\infty \1_{\{s<T\}} \sqrt{\langle f(X_s),-\alpha_s\rangle}\,\rd s\Bigr]\\
				&\le C\sqrt \rho \, \int_{t_0}^\infty \sqrt{-\dot \Phi_s}\,\rd s= C  \sqrt {\frac{\rho C''}{2\theta-1}} \, \int_{t_0}^\infty (t-t_0+1)^{-\frac{\theta}{2\theta-1}}\,\rd s \\
				&=   C \sqrt{\rho C''}  \frac{ \sqrt {2\theta-1 }}{1-\theta} <\infty.
			\end{align*}
			This implies almost sure convergence of $(\int_{t_0}^t \alpha_s \, \rd s)_{t \ge t_0}$, on  $\{\bar T_{\rho,C,t_0,U}=\infty\}$.
			Note that also  $(\int_{t_0}^\infty \beta_t \, \rd W_t)_{t\ge t_0}$ converges almost surely on $\{\bar T_{\rho,C,t_0,U}=\infty\}$, 
			since
			$$
				\int_{t_0}^{\bar T_{\rho,C,t_0,U}} |\beta_s|_F^2 \, \rd s\le \int_{t_0}^\infty C \sigma_s^2 \, \rd s <\infty.
			$$
			Hence, we have almost sure convergence of
			$$
				X_t-X_{t_0} = \int_{t_0}^t \alpha_t \, \rd t + \int_{t_0}^t \beta_t \, \rd W_t,
			$$
			on $\{\bar T_{\rho,C,t_0,U}=\infty\}$.
		
		1.c) Combining 1.a) and 1.b) we conclude that on $\{T_{\rho,C,t_0,U}=\infty\}$ we have, almost sure, convergence of $(X_t)_{t \ge 0}$.
		Indeed, by 1.a) we have that up to nullsets $\{T_{\rho,C,t_0,U}=\infty\}$ is contained in 
		$$
		\bigcup_{n\in\N} \{\bar T_{\rho,C,t_0+n,U}=\infty\}
		$$
		and by 1.b) on each of the latter sets $(X_t)_{t \ge 0}$ converges almost surely.

		2) Now let $K\subset \R^d$ be a compact subset.
		We show that we have almost sure convergence of $(X_t)_{t \ge 0}$ on $\{T_{\rho,C,t_0,K}=\infty\}$. The statement then follows by observing that 
		$$
		\IC \cap \IL\cap \bigl\{ \limsup_{t \to \infty} \frac{|\beta_t|}{\sigma_t}<\infty \bigr\}=\bigcup_{n\in\N} \{T_{\frac 1n,n,n,[-n,n]^d}=\infty\}.
		$$
		
		By Lemma~\ref{lem:lojaset},  $K$ contains finitely many critical levels, say $\ell_1,\dots,\ell_m\in\R$, and there exists $\eps>0$, $\L>0$, $\theta\in[\frac \sigma{2\sigma-1},1)$ and bounded open neighbourhoods $U_k\supset F^{-1}((\ell_k-\eps,\ell_k+\eps))\cap K$ such that for each $k=1,\dots, m$ and $y\in U_k$,
		$$
		|f(y)|\ge \L |F(y)-\ell_k|^{\theta}.
		$$
		
		By Theorem~\ref{theo1}, we have that on $\{T_{\rho,C,t_0,K}=\infty\}$, almost surely, the limit $\lim_{t\to\infty} F(X_t)$ exists and $\lim_{t\to\infty} f(X_t)=0$. Since $\inf_{y\in K\backslash (U_1\cup\ldots \cup U_m)} |f(y)|>0$, we get that on $\{T_{\rho,C,t_0,U}=\infty\}$, almost surely, $(X_t)_{t \ge 0}$ attains from a random time onwards only values in $U_1\cup\ldots\cup U_m$ and using the \loja-inequality we get that $(F(X_t))_{t \ge 0}$ converges to one of the  levels $\ell_k$ and, in particular, attains from a random time onwards only values in a $U_k$ (which may be random). Hence, up to nullsets, $\{T_{\rho,C,t_0,K}=\infty\}$ is contained in
		$$
		\bigcup_{n\in\N}\bigcup_{k=1}^m \{T_{\rho,C,t_0+n,U_k}=\infty\}.
		$$
		By 1.c) we have almost sure convergence of $(X_t)_{t \ge 0}$ on each of the countably many latter events.
	\end{proof}

\section{Optimality of Theorem~\ref{theo2}} \label{sec:optimal}

In this section, we provide an example which satisfies the assumptions of Theorem~\ref{theo2} when choosing $\sigma_t=(t+1)^{-1}$, but for which $(X_t)_{t\ge0}$ does not converge.

We consider a rotationally invariant Lyapunov function $F$ on $\R^2$. Let 
$$
\psi:[0,\infty)\to[0,\infty), r\mapsto
\begin{cases}
(1-r)^2, & \text{if } r\ge \frac 12, \\
\frac 83 r^3-3r^2+\frac 23, & \text{else},
\end{cases}
$$
and set  $F:\R^2\to\R$, $F(x):=\psi(|x|)$. 
Note that $F$ is a $C^2$-function with critical points 
$$\cC:=f^{-1}(\{0\})= \{x\in\R^2: |x|\in\{0,1\}\},$$
where we again denote by $f:=\nabla F$ the gradient of $F$.
For every $(x_1,x_2)\in\R^2$ we let 
$$
\mathrm{ort}(x_1,x_2):= (-x_2,x_1).
$$
and denote by $\varphi:\R^2\to [0,1]$  a rotationally invariant $C^{\infty}$-function with
$$
\varphi|_{B(0,2)\backslash B(0,\frac { 2}{ 3})}\equiv 1\text{ \ and \ }\varphi|_{B(0,\frac {1}{ 2})\cup B(0,3)^c}\equiv 0.
$$
Consider the SDE
$$
dX_t= \Bigl(-f(X_t)+ \mathrm{ort}(X_t) \,  ||X_t|-1| \Bigr)\, dt + \frac 1{t+1} \varphi(X_t)\,d W_t
$$
started in a point $x_0\in\R^2\backslash\{0\}$,  where $(W_t)_{t\ge 0}$ is a $2$-dimensional Brownian motion with initial value $W_0=0$.

\begin{theorem} \label{thm:counter}
\begin{enumerate}
\item[(i)] 	$F$ is a \loja-function.
\item[(ii)] 	$\P(\IC \cap \IL)=1$.
\item[(iii)] Almost surely, $(X_t)_{t \ge 0}$ does not converge.
\end{enumerate}
\end{theorem}


\begin{figure}
	\begin{minipage}[b]{.25\linewidth} 
		\includegraphics[width=\linewidth]{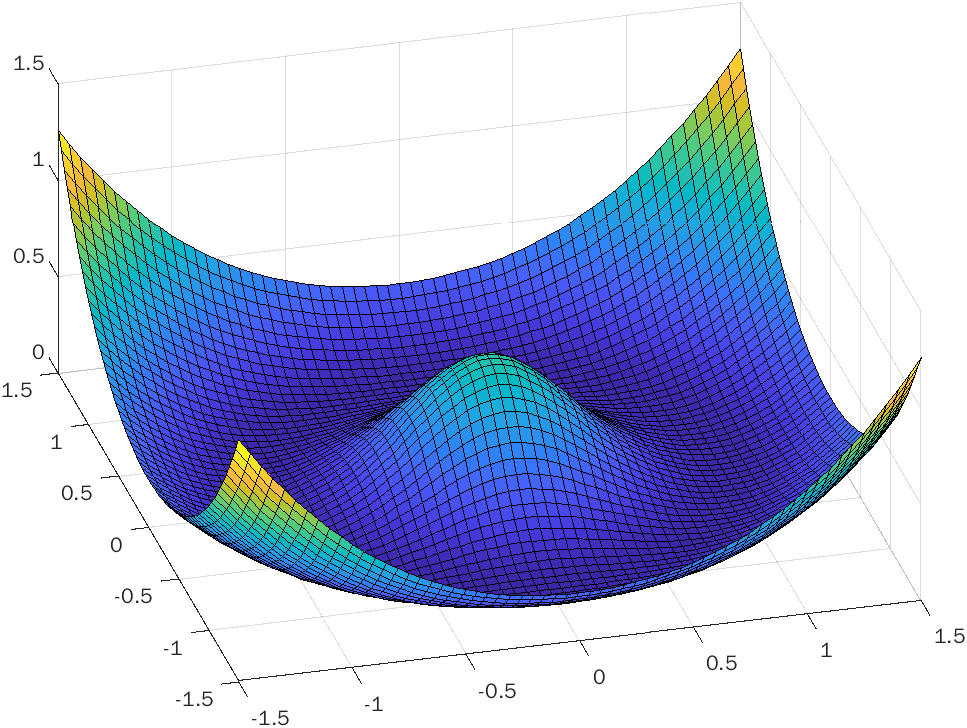}
		\caption{(a) Graph of $F$}
	\end{minipage}
	\hspace{.05\linewidth}
	\begin{minipage}[b]{.25\linewidth} 
		\includegraphics[width=\linewidth]{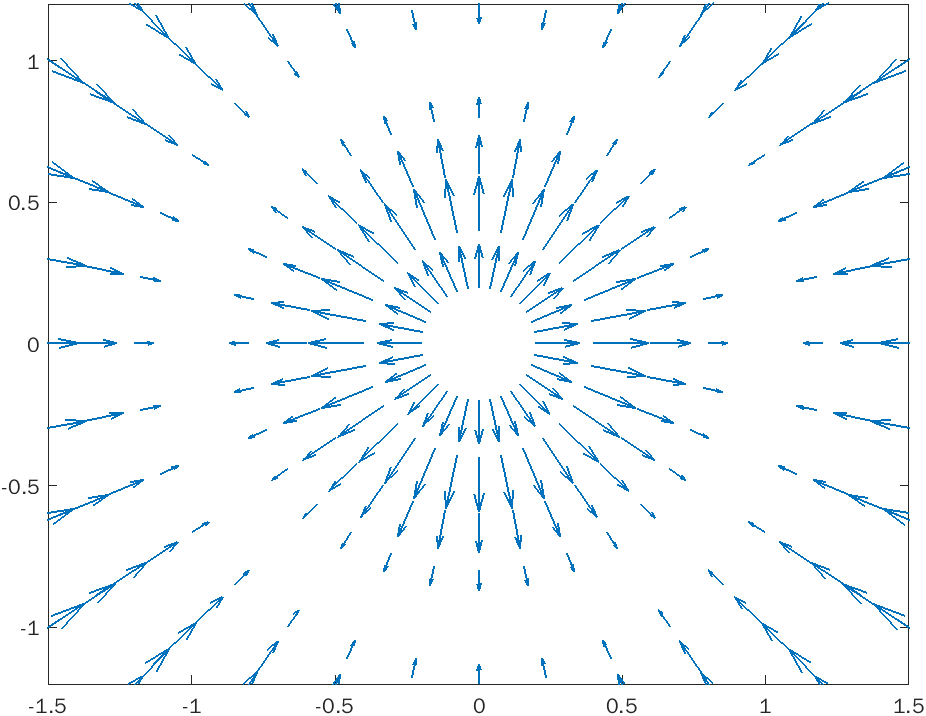}
		\caption{(b) Negative gradient field}
	\end{minipage}
	\hspace{.05\linewidth}
	\begin{minipage}[b]{.25\linewidth} 
		\includegraphics[width=\linewidth]{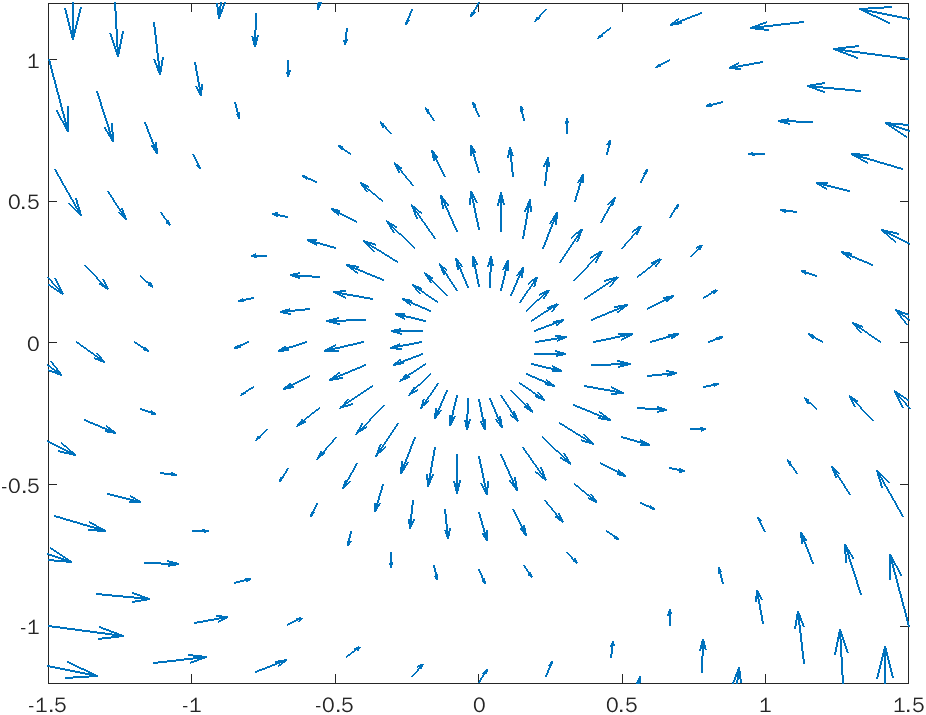}
		\caption{(c) Vectorfield of the drift}
	\end{minipage}
\end{figure}

The proof is based on the following proposition.

\begin{prop}\label{prop:7435}	For $t\ge 0$ consider the stopping time
	$$
	T_t:=\inf\bigl\{s\ge t: X_s\not \in B(0,2)\backslash B(0,\sfrac {2}{ 3})\bigr\}.
	$$
	\begin{enumerate}
		\item[(i)] $\P(\IC \cap \IL)=1$ and, almost surely, $\lim_{t\to\infty} |X_t|=1$. In particular, almost surely, for all  sufficiently large $t$, $T_t=\infty$.
		\item[(ii)] For every $t\ge0$, $(\bar Z_s)_{s\ge t}:=(|X_s|-1)_{s \ge t}$ solves on the random time interval $[t,T_t]$ the SDE
		$$
		d\bar Z_s=\Bigl(-2 \bar Z_s+\frac 12 \frac1{\bar Z_s+1}\frac 1{(s+1)^2}\Bigr)\,ds +\frac 1{s+1}dB_s
		$$
		where $B_s:=\int_t^s  \frac 1{|X_u|} \langle X_u, dW_u\rangle$ is a Brownian motion.
		\item[(iii)] For $t\ge 0$ let $(Z_s^{(t)})_{s\ge t}$ denote the solution to 
		\begin{align}\label{eq83}
		Z^{(t)}_t=\bar Z_t\text{ \ and \ } dZ^{(t)}_s=-2 Z^{(t)}_s\,ds +\frac 1{s+1}dB_s
		\end{align}
		Then,
		$$
		\int_t^{T_t} |\bar Z_s- Z^{(t)}_s|\, ds \le \frac 12 \frac 1{t+1}.
		$$
	\end{enumerate}
\end{prop}

\begin{proof} (i):  For the choice
		$$
			\alpha_t := -f(X_t)+ \mathrm{ort}(X_t) \,  ||X_t|-1| \text{ \ and \ }
			\beta_t:= \frac 1{t+1} \varphi(X_t) \,\II_2
		$$
		 $(X_t)_{t\ge 0}$ solves the SDE~(\ref{SDE}).  We verify that the events $\IB$ and $\IC$ are almost sure events. Note that the diffusivity on $B(0,\frac 12)\cup B(0,3)^c$ is zero and that the vector field pushes $(X_t)_{t \ge 0}$ towards values in $\overline {B(0,3)}\backslash B(0,\frac12)$. Hence, from a deterministic time, say $t_0$,  onwards the process $(X_t)_{t \ge 0}$ stays on the compact set  $\overline {B(0,3)}\backslash B(0,\frac12)$. On the latter compact set $F$, $\nabla F$ and the Hessian $H$ of $F$ are uniformly bounded since $F$ is $C^2$. Therefore, $\IB$ is an almost sure set.
		 
		 For $t\ge t_0$, by orthogonality of $f$ and $\mathrm{ort}$ one has
		 $$
		 \langle f(X_t),-\alpha_t\rangle =|f(X_t)|^2
		 $$
		 and, additionally,
		 $$
		 |\alpha_t|=|-f(X_t)+\mathrm{ort}(X_t) ||X_t|-1|| \le \underbrace{|f(X_t)|}_{=2|1-|X_t||} +\underbrace{|X_t|}_{\le 3}\,||X_t|-1|\le 5|1-|X_t||=\frac 52 |f(X_t)|
		 $$ 
		 Moreover, 
		 $$
		 \int_{0}^\infty |\beta_s|_F^2\,ds \le \int_0^\infty \frac 1{(s+1)^2} |\II_2|_F^2\,ds =2  \int_0^\infty \frac 1{(s+1)^2}\,ds<\infty.
		 $$
		 Consequently, $\IC$ is an almost sure event. By Theorem~\ref{theo1}, $(f(X_t))_{t\ge0}$ converges, almost surely, to zero so that $\lim_{t\to\infty}|X_t|=1$, almost surely.

	(ii): 	Note that $|\cdot|: \R^2 \setminus\{0\} \to \R$ is $C^\infty$ and, on $[t, T_t]$, $-f(X_t)=-2\bar Z_t \frac{X_t}{|X_t|}$ as well as $\langle X_t, \mathrm{ort}(X_t)\rangle =0$, so that with the Itô-formula we get 
	\begin{align*}
		d\bar Z_s &= \frac{1}{|X_s|} \langle X_s, dX_s\rangle +\frac 12 \frac{1}{(s+1)^2} \frac{1}{|X_s|}ds \\
		&= \Bigl(-2 \bar Z_s+\frac 12 \frac1{\bar Z_s+1}\frac 1{(s+1)^2}\Bigr)\,ds +\frac 1{s+1}dB_s.
	\end{align*}

	(iii): Fix $t\ge0$ and consider $(\Upsilon^{(t)}_s)_{s \ge t}:=(\bar Z_s- Z^{(t)}_s)_{s\ge t}$. Then on $[t,T_t]$
	$$
	d\Upsilon^{(t)}_s=-2\Upsilon_s^{(t)} \,ds+\frac 12 \frac1{\bar Z_s+1}\frac 1{(s+1)^2}\,ds.
	$$ 
	With $\Upsilon^{(t)}_t=0$ we thus get that for $t\le s\le T_t$
	$$
	\Upsilon_s^{(t)}=\frac 12 \int_t^s e^{-2(s-u)}   \frac1{\bar Z_u+1}\frac 1{(u+1)^2}  \, du\le  \int_t^s e^{-2(s-u)}  \frac 1{(u+1)^2}  \, du,
	$$
	where we used that for  $t\le u\le T_t$, $ \frac1{\bar Z_u+1}\le 2$. We thus get with Fubini that 
	\begin{align*}
	\int_{t}^{T_t} |\bar Z_s-Z_s^{(t)}|\, ds&\le \int_t^\infty \int_t^s  e^{-2(s-u)}  \frac 1{(u+1)^2}  \, du\, ds\\
	&= \int_t^\infty \underbrace{ \int_u^\infty  e^{-2(s-u)}    \, ds}_{=\frac 12}\, \frac 1{(u+1)^2}\, du
	=\frac 12 \frac 1{t+1}.
	\end{align*}
\end{proof}

Now, let $t\ge 0$ and $x_0\in\R$ and let $(Z_s)_{s \ge t}$ be the solution of the SDE
\begin{align}\label{eq83745}
\rd Z_s = -2 Z_s \, \rd s + (s+1)^{-1}\, \rd B_s, \text{ with  }Z_{t}=x_0.
\end{align}
Using Itô's lemma and the Dubins-Schwarz theorem it is straight-forward to see that one can represent the solution as
$$
(Z_s)_{s \ge t} = (e^{-2s}\tilde B_{g_s})_{s \ge t},
$$
where $(\tilde B_s)_{s \ge g_t}$ is a Brownian motion started at time $g_{t}$  in  $\tilde B_{g_{t}} =e^{2t} x_0$ and $(g_s)_{s\ge0}$ is given  by
$$
g_s:= \int_0^s (u+1)^{-2}e^{4u}\,\rd u \sim \frac 14 (s+1)^{-2} e^{4s}.
$$
Note that $g_s\le \frac{e^{4s}}{(s+1)^2}$.

\begin{lemma}\label{le:3531}For every $t\ge 0$ and $x_0\in\R$ the solution $(Z_s)_{s \ge t}$ of~(\ref{eq83745}) satisfies
	$$
	\int_{t}^\infty |Z_s|\, \rd s=\infty, \qquad \text{almost surely.}
	$$
\end{lemma}

\begin{proof} We let for $s\ge t$, $\bar B_s:=\tilde B_{g_s}-e^{2t} x_0$ and  note that 
	$$
	\Bigl|\int_{t}^u|Z_s|\,\rd s - \int_{t}^u |e^{-2s}\bar B_s|\,\rd s\Bigr| \le |x_0| \int_{0}^{\infty}e^{-2s}\,\rd s=\sfrac 12 |x_0|. 
	$$
	Hence, it suffices to show that $\int_{t}^\infty|e^{-2s}\bar B_s|\,\rd s=\infty$, almost surely.
	Set 
	$$\kappa := \int_\R \frac1{\sqrt{2\pi}} |x|\,e^{-x^2/2}\,\rd x=\sqrt{\frac2\pi}.$$
	One has
	\begin{align*}
	\E\Bigl[\int_{t}^u e^{-2s} |\bar B_s| \, \rd s \Bigr]&= \int_{t}^u e^{-2s} \E[|\bar B_{s}|] \ \rd s = \kappa \int_{t}^u e^{-2s}\sqrt{g_s-g_{t}} \, \rd s .
	\end{align*}
	Using that $g_s\sim\frac 14 (s+1)^{-2} e^{4s}$ we get that
	$$
		\kappa \int_{t}^u e^{-2s}\sqrt{g_s-g_{t}} \, \rd s\sim \sfrac 12 \kappa\int_{t}^u (s+1)^{-1}\,\rd s\sim \sfrac 12 \kappa\log u.
	$$
	%
	To estimate the variance we first show that for  $v\ge u \ge t$ we have
	\begin{align}\label{eq8476}
	\cov(|\bar B_u|, |\bar B_v|)\le g_v- g_{t}.
	\end{align}
	Indeed,
	\begin{align*}
	\E[|\bar B_v \bar B_u|] &\le \E[\bar B_u^2] + \E[|\bar B_v-\bar B_u|]\, \E[|\bar B_u|] \\
	&\le g_u-g_{t} + \kappa^2  \sqrt{g_v-g_u} \sqrt{g_u-g_{t}}
	\end{align*}
	and
	$$
	\E[|\bar B_v|] \,\E[|\bar B_u|]= \kappa^2 \sqrt{g_v-g_{t}} \sqrt{g_u-g_{t}}.
	$$
	Thus (\ref{eq8476}) follows since $g_u\ge g_{t}$. We conclude that for all $\ell\ge t$	
	\begin{align*}
	\var\Bigl(\int_{t}^\ell e^{-2s} |\bar B_s| \, \rd s \Bigr) &= \int_{t}^\ell \int_{t}^\ell e^{-2(u+v)}\cov(|\bar B_u|,|\bar B_v|) \ \rd u \,  \rd v \\
	&  \le 2 \int_{t}^\ell  \int_u^\ell e^{-2(u+v)} g_u \ \rd v \,  \rd u  \\
	& \le  \int_0^\ell  e^{-4u} g_u \,  \rd u  \le \int_0^\ell  \frac{1}{(u+1)^2} \,  \rd u \le 1.
	\end{align*}
	Hence, as a consequence of the Chebyshev inequality one gets that
	$$
	\int_{t}^\infty e^{-2s} |\bar B_s| \, \rd s =\lim_{u\to\infty} \int_{t}^u e^{-2s} |\bar B_s| \, \rd s =\infty,\quad\text{almost surely.}
	$$		
\end{proof}	
\black

\begin{proof}[Proof of Theorem~\ref{thm:counter}]
(i): We first verify that $F$ is a \loja-function. 
We  need to analyse the critical points only. One has 
$$
\Hess F(0)= \begin{pmatrix}
	-6 & 0 \\ 0 & -6
\end{pmatrix}
$$
so that the Hessian has full rank at $0$ which implies validity of a \loja-inequality with $\theta=\frac 12$ on an appropriate neighbourhood of $0$. Next, let $x\in\R^2$ with $|x|=1$. Since  $\psi''(1)=2$,  $\psi$ satisfies a \loja-inequality
$$
|\psi'(r)|\ge \L |\psi(r)-\psi(1)|^{\frac 12}
$$
for all $r\in[1-\eps,1+\eps)$ and appropriately fixed $\eps\in(0,\frac 12]$ and $\L>0$. Consequently, for $y\in B(0,1+\eps)\backslash B(0,1-\eps)$ one has 
$$
|f(y)|=|\psi'(|y|)|\ge \L\, |\psi(|y|)-\psi(1)|^{1/2} =\L\, |F(y)-F(x)|^{1/2}.
$$

(ii): has been shown in Prop.~\ref{prop:7435}.

(iii): We conceive $(X_t)_{t \ge 0}$ as complex-valued process by letting
$$
Y_t:= X^{(1)}_t+i X^{(2)}_t
$$
and note that (since $(Y_t)_{t \ge 0}$ does not hit $0$) there is a continuous adapted process  $(\Phi_s)_{s \ge 0}$  satisfying
	$$
		Y_s = |X_s| e^{i \Phi_s}.
	$$
	Now on $[t,T_t]$
		\begin{align*}
		d \Phi_s =  ||X_s|-1| \, d s + \frac 1{|X_s|^2}  \frac{1}{s+1} \langle \mathrm{ort}(X_s),  dW_s \rangle
	\end{align*}
	so that 
	$$
	d\langle \Psi\rangle_s=\frac 1{|X_s|^2} \frac1{(s+1)^2}\, ds.
	$$
	We thus get that 
	$$
	\int_{t}^{T_t} d\langle \Psi\rangle_s =  \int_{t}^{T_t} \frac 1{|X_s|^2} \frac1{(s+1)^2} \,ds\le \sfrac 94 \int_t^\infty \frac1{(s+1)^2} \,ds<\infty
	$$
	so that $\int_t^{T_t\wedge s} \frac 1{|X_u|^2}  \frac{1}{u+1} \langle \mathrm{ort}(X_u),  dW_u \rangle$ converges almost surely as $s\to \infty$, say to $\bar \Psi^{(t)}$. Consequently, on $\{T_t=\infty\}$ one has as $s\to \infty$
	$$
	\Phi_s =\int_t^s  ||X_u|-1| \, d u + \bar \Phi^{(t)}+o(1),
	$$
	where $o(1)$ stands for a term converging  to zero as $s\to\infty$. By Proposition~\ref{prop:7435} (iii) and  Lemma~\ref{le:3531}, $\int_t^s  ||X_u|-1| \, d u\to\infty$ as $s\to\infty$ which shows that $(X_t)_{t \ge 0}$ does not converge on $\{T_t=\infty\}$. The result follows by noticing that $\bigcup_{t\in\N} \{T_t=\infty\}$ is an almost sure set by Proposition~\ref{prop:7435} (i).
%
%
%
	\end{proof}

\appendix
\section{Criterion for staying local}  \label{sec:app}
	Assume that $F$ satisfies $\lim_{|x| \to \infty} F(x)=\infty$ and consider the set 
	$$
	\IL_1 := \Bigl\{ \sup\limits_{t \ge 0} \int_0^t \bigl(\langle f(X_s),\alpha_s \rangle + \frac 12 \tr (\beta_s^\dagger H(X_s) \beta_s) \bigr) \, \rd s <\infty \Bigr\}.
	$$
	We will show that $\limsup_{t \to \infty} |X_t|<\infty$, almost surely, on $\IL_1$. By Itô's formula
	$$
	\rd F(X_t)= \langle f(X_t) ,\alpha_t\rangle \,\rd t+   \langle f(X_t),\beta_t\,\rd W_t\rangle +\frac 12 \tr(\beta_t^\dagger H(X_t) \beta_t)\,\rd t.
	$$
	Now, let $n \in \N$ and consider the stopping time 
	$$
	T_n := \inf \Bigl\{ t \ge 0 : \int_0^t \bigl( \langle f(X_s),\alpha_s \rangle + \frac 12 \tr (\beta_s^\dagger H(X_s) \beta_s) \bigr) \, \rd s \ge n \Bigr\}.
	$$
	and $(Y_t)_{t \ge 0}$ given by
	$$
	Y_t := F(X_t)- \int_0^t \bigl( \langle f(X_s),\alpha_s \rangle + \frac 12 \tr (\beta_s^\dagger H(X_s) \beta_s) \bigr) \, \rd s.
	$$
	Then, $(Y_t^{T_n})_{t \ge 0}$ is a local martingale that is bounded from below and, thus, a supermartingale. With the martingale convergence theorem we get almost sure convergence of $(Y_t^{T_n})_{t \ge 0}$ and, thus, on $\{T_n = \infty\}$ we have that $\sup_{t\ge0} F(X_t)$ is finite, almost surely. Finally, note that
	$$
	\IL_1 = \bigcup\limits_{n \in \N} \{ T_n = \infty\},
	$$
	so that $\limsup_{t \to \infty} |X_t|<\infty$, almost surely, on $\IL_1$. Note that for the choice $(\alpha_t)_{t\ge0} = (-f(X_t))_{t \ge 0}$ we have
	$$
	\Bigl\{ \sup\limits_{t \ge 0}|H(X_t)|<\infty \text{ and } \int_0^\infty |\beta_s|_F^2 \, \rd s<\infty \Bigr\} \subset \IL_1
	$$
	as well as for a compact set $K \subset \R^d$
	\begin{align*}
		\Bigl\{ &\1_{\{X_t \notin K\}} (-|f(X_t)|^2+\frac 12 \tr ((\beta_t)^\dagger H(X_t) \beta_t) )\le 0  \text{ for all large } t \\ 
		&\text{ and } \int_0^\infty \1_{\{X_s \in K\}}|\beta_s|_F^2 \, \rd s <\infty \Bigr\}\subset \IL_1.
	\end{align*}

{\bf Acknowledgement.}
The authors would like to thank the two anonymous referees for their valuable comments.
Funded by the Deutsche Forschungsgemeinschaft (DFG, German Research Foundation) under Germany's Excellence Strategy EXC 2044--390685587, Mathematics Münster: Dynamics--Geometry--Structure.


\begin{thebibliography}{AMA05}

\bibitem[AMA05]{AMA05}
P.-A. Absil, R.~Mahony, and B.~Andrews.
\newblock Convergence of the iterates of descent methods for analytic cost
  functions.
\newblock {\em SIAM J. Optim.}, 16(2):531--547, 2005.

\bibitem[Ben99]{benaim1999dynamics}
M.~Bena{\"\i}m.
\newblock Dynamics of stochastic approximation algorithms.
\newblock In {\em Seminaire de probabilites XXXIII}, pages 1--68. Springer,
  1999.

\bibitem[CHS87]{chiang1987diffusion}
T.-S. Chiang, C.-R. Hwang, and S.~J. Sheu.
\newblock Diffusion for global optimization in {$\mathbb R^n$}.
\newblock {\em SIAM Journal on Control and Optimization}, 25(3):737--753, 1987.

\bibitem[CMI14]{colding2014}
T.~Colding and W.~Minicozzi~II.
\newblock Lojasiewicz inequalities and applications.
\newblock In {\em Surveys in Differential Geometry, {XIX}}, pages 63--82. 2014.

\bibitem[Cur44]{curry1944method}
H.~B. Curry.
\newblock The method of steepest descent for non-linear minimization problems.
\newblock {\em Quarterly of Applied Mathematics}, 2(3):258--261, 1944.

\bibitem[DK21]{dereich2021convergence}
S.~Dereich and S.~Kassing.
\newblock Convergence of stochastic gradient descent schemes for
  {{\L}}ojasiewicz-landscapes.
\newblock arXiv:2102.09385, 2021.

\bibitem[FT21]{fournier2021simulated}
N.~Fournier and C.~Tardif.
\newblock On the simulated annealing in {$\Bbb{R}^d$}.
\newblock {\em J. Funct. Anal.}, 281(5):Paper No. 109086, 30, 2021.

\bibitem[FW12]{freidlin2012random}
M.~I. Freidlin and A.~D. Wentzell.
\newblock {\em Random Perturbations of Dynamical Systems}.
\newblock Grundlehren der mathematischen Wissenschaften. Springer New York,
  2012.

\bibitem[Gar85]{gardiner1985handbook}
C.~W. Gardiner.
\newblock {\em Handbook of stochastic methods}, volume~3.
\newblock springer Berlin, 1985.

\bibitem[GH86]{geman1986diffusions}
S.~Geman and C.-R. Hwang.
\newblock Diffusions for global optimization.
\newblock {\em SIAM Journal on Control and Optimization}, 24(5):1031--1043,
  1986.

\bibitem[Gil00]{gillespie2000chemical}
D.~T. Gillespie.
\newblock The chemical {L}angevin equation.
\newblock {\em The Journal of {C}hemical {P}hysics}, 113(1):297--306, 2000.

\bibitem[Har12]{haraux2012some}
A.~Haraux.
\newblock Some applications of the {{\L}}ojasiewicz gradient inequality.
\newblock {\em Communications on Pure \& Applied Analysis}, 11(6):2417, 2012.

\bibitem[HKS89]{holley1989asymptotics}
R.~A. Holley, S.~Kusuoka, and D.~W Stroock.
\newblock Asymptotics of the spectral gap with applications to the theory of
  simulated annealing.
\newblock {\em Journal of functional analysis}, 83(2):333--347, 1989.

\bibitem[JKO98]{jordan1998variational}
R.~Jordan, D.~Kinderlehrer, and F.~Otto.
\newblock The variational formulation of the fokker--planck equation.
\newblock {\em SIAM journal on mathematical analysis}, 29(1):1--17, 1998.

\bibitem[{\L}oj63]{lojasiewicz1963propriete}
S.~{\L}ojasiewicz.
\newblock Une propri{\'e}t{\'e} topologique des sous-ensembles analytiques
  r{\'e}els.
\newblock {\em Les {\'e}quations aux d{\'e}riv{\'e}es partielles}, 117:87--89,
  1963.

\bibitem[{\L}oj65]{lojasiewicz1965ensembles}
S.~{\L}ojasiewicz.
\newblock Ensembles semi-analytiques.
\newblock {\em Lectures Notes IHES (Bures-sur-Yvette)}, 1965.

\bibitem[{\L}oj84]{lojasiewicz1984sur}
S.~{\L}ojasiewicz.
\newblock {Sur les trajectoires du gradient d'une fonction analytique.
  (Trajectories of the gradient of an analytic function)}.
\newblock {\em {Semin. Geom., Univ. Studi Bologna}}, 1982/1983:115--117, 1984.

\bibitem[LTE17]{WeinanE2017}
Q.~Li, C.~Tai, and W.~E.
\newblock Stochastic modified equations and adaptive stochastic gradient
  algorithms.
\newblock In {\em Proceedings of the 34th International Conference on Machine
  Learning}, volume~70, pages 2101--2110. PMLR, 06--11 Aug 2017.

\bibitem[PdM82]{palis2012geometric}
J.~Palis, Jr. and W.~de~Melo.
\newblock {\em Geometric theory of dynamical systems: an introduction}.
\newblock Springer Science \& Business Media, 1982.

\bibitem[SSG14]{schnoerr2014complex}
D.~Schnoerr, G.~Sanguinetti, and R.~Grima.
\newblock The complex chemical {L}angevin equation.
\newblock {\em The Journal of Chemical Physics}, 141(2):07B606\_1, 2014.

\bibitem[vK92]{van1992stochastic}
N.~G. van Kampen.
\newblock {\em Stochastic processes in physics and chemistry}, volume~1.
\newblock Elsevier, 1992.

\end{thebibliography}

\end{document}